\documentclass[11pt,a4paper]{amsart}
\usepackage{amsmath,amsfonts,amsthm,amsopn,color,amssymb,enumitem}
\usepackage[a4paper]{geometry}
\usepackage{palatino}
\usepackage{graphicx}
\usepackage[colorlinks=true]{hyperref}
\hypersetup{urlcolor=blue, citecolor=red, linkcolor=blue}

\usepackage{relsize}
\usepackage{esint}
\usepackage{verbatim}
\usepackage{mathrsfs}
\usepackage{xcolor}
\usepackage{tikz}

\newcommand{\RN}{{\mathbb{R}^N}}

\newcommand{\RQ}{{\mathbb{R}^4}}

\newcommand{\QM}{4-\frac{\mu}{2}}
\newcommand{\TM}{3-\frac{\mu}{2}}
\newcommand{\DM}{2-\frac{\mu}{2}}
\newcommand{\DQ}{2-\frac{\mu}{4}}
\newcommand{\UQ}{1-\frac{\mu}{4}}

\newcommand{\Ud}{{\mathcal{U}_{\lambda, \xi}}}
\newcommand{\Uj}{{\mathcal{U}_{\lambda, \xi_j}}}
\newcommand{\Ui}{{\mathcal{U}_{\lambda, \xi_i}}}
\newcommand{\U}{{\mathcal{U}_{\lambda, \xi_1}}}

\newcommand{\Zo}{{Z^0_{\lambda, \xi_1}}}
\newcommand{\Zj}{{Z^0_{\lambda, \xi_j}}}
\newcommand{\beq}{\begin{equation}}
\newcommand{\eeq}{\end{equation}}

\newtheorem{theorem}{Theorem}[section]
\newtheorem*{theorem*}{Theorem}
\newtheorem{lemma}[theorem]{Lemma}

\newtheorem{proposition}[theorem]{Proposition}

\theoremstyle{definition}
\newtheorem{remark}[theorem]{Remark}

\title[Choquard critical system]{Segregated solutions for a critical Choquard system with a small interspecies repulsive force}

\author{Sabrina Caputo}
\address[S. Caputo]{Dipartimento di Matematica, Universit\`a degli studi di Bari Aldo Moro, via Edoardo Orabona 4,70125 Bari, Italy}
\email{s.caputo28@phd.uniba.it}

\date\today
\keywords{Choquard equation, coupled elliptic systems, blow-up solutions.\\
Work supported by "INDAM-GNAMPA Project", CUP E53C25002010001.}

\begin{document}
\begin{abstract}
    We consider the following elliptic system 
    \begin{equation*}
        -\Delta u_i=(|x|^{-\mu}*|u_i|^{4-\frac{\mu}{2}})|u_i|^{2-\frac{\mu}{2}}u_i+\sum_{j=1\,  j\not =i}^2\beta u_iu_j^2\quad \hbox{in} \, \, \RQ,\quad i=1,2.
    \end{equation*}
    If $\beta<0, |\beta|$ is small enough we build solutions such that the first component looks like the radial positive solution of the single equation, while the second one blows-up at the $k$ vertices of a regular polygon.
\end{abstract}
\maketitle
\section{Introduction}
In this paper we will focus on a class of coupled elliptic systems with nonlocal nonlinearities of Choquard type. Specifically, we will study the following system in $\RQ$
\beq\label{Pe}
\left\{\begin{aligned} &-\Delta u =(|x|^{-\mu}*|u|^{4-\frac{\mu}{2}})|u|^{2-\frac{\mu}{2}}u+\beta uv^2\quad &\hbox{in}\,\, \RQ\\
&-\Delta v =(|x|^{-\mu}*|v|^{4-\frac{\mu}{2}})|v|^{2-\frac{\mu}{2}}v+\beta u^2v \quad &\hbox{in}\,\, \RQ
\end{aligned}\right.\eeq
where $*$ stands for the standard convolution, $\beta<0,$ and $0<\mu\leq 4.$\\
A more general version of \eqref{Pe} is the following critical system in $\RN$
\beq\label{PO}
\left\{\begin{aligned} &-\Delta u =(|x|^{-\mu}*|u|^{2^*_\mu})|u|^{2^*_\mu-2}u+\beta u^{\frac{2^*}{2}-1}v^{\frac{2^*}{2}}\quad &\hbox{in}\,\, \RN\\
&-\Delta v =(|x|^{-\mu}*|v|^{2^*_\mu})|v|^{2^*_\mu-2}v+\beta u^{\frac{2^*}{2}}v^{\frac{2^*}{2}-1} \quad &\hbox{in}\,\, \RN
\end{aligned}\right.\eeq
where $2^*:=\frac{2N}{N-2}$ is the critical Sobolev exponent and $2^*_\mu:=\frac{2N-\mu}{N-2}$ is the upper critical exponent in the sense of the Hardy-Littlewood-Sobolev inequality \cite{LL}.\\

The single equation of \eqref{PO}, often referred to as the \textbf{nonlinear Choquard} or \textbf{Schr\"odinger-Newton equation}, has deep roots in mathematical physics. Indeed, it seems to originate from the model of the polaron of H. Fr\"ohlich and S. Pekar, where free electrons in an ionic lattice interact with phonons associated with the deformation of the lattice or with the polarization that it creates in the medium (interaction of an electron with its own hole)\cite{F},\cite{FF},\cite{P}. It was later utilized by Choquard in 1976 in the modeling of a one-component plasma. More recently, this kind of equation has gained renewed interest as the non-relativistic limit of the Einstein-Klein-Gordon system, often appearing in the study of boson stars and scalar field dark matter \cite{RB}.\\

The mathematical framework for such nonlocal problems has been extensively developed in \cite{MV}, who established the existence results and qualitative properties for ground state solutions. In particular, the critical regime presents formidable mathematical challenges. It sits at the threshold of the Hardy-Littlewood-Sobolev inequality (see \cite{LL}).
\begin{proposition}
    Let $t, r >1$ and $0<\mu<N$ with $\frac{1}{t}+\frac{\mu}{N}+\frac{1}{r}=2.$ There exists a sharp constant $C=C(t,N,\mu,r)$ such that for any $f\in L^t(\RN)$ and $h\in L^r(\RN)$ 
    $$\int_{\RN}\int_{\RN}\frac{|f(x)||h(y)|}{|x-y|^\mu}\, dx\, dy  \leq C\|f\|_t\|h\|_r.$$  
\end{proposition}
Notice that, by Hardy-Littlewood-Sobolev inequality, the integral 
$$\int_{\RN}\int_{\RN}\frac{|u(x)|^q|u(y)|^q}{|x-y|^\mu}\, dx\, dy $$ is well-defined if $|u|^q\in L^t(\RN)$ for some $t>1$ satisfying $\frac{2}{t}+\frac{\mu}{N}=2.$ Thus, for $u\in H^1(\RN),$ by Sobolev embedding theorems, we know that
$$2\leq tq\leq 2^*=\frac{2N}{N-2}$$ namely $$\frac{2N-\mu}{N}\leq q\leq \frac{2N-\mu}{N-2}.$$ Thus, $\frac{2N-\mu}{N}$ is called the lower critical exponent, while $2^*_\mu:=\frac{2N-\mu}{N-2}$ is the upper critical exponent in the sense of the Hardy-Littlewood-Sobolev inequality.\\

For the upper critical case, by using the moving plane methods in integral developed in \cite{CL,CLO}, Lei \cite{LEI}, Du and Yang \cite{DY}, Guo et al. \cite{GH} classified independently the positive solutions of the critical Hartree equation
\begin{equation}\label{plim}
    -\Delta u=(|x|^{-\mu}*|u|^{2^*_\mu})|u|^{2^*_\mu-2}u, \quad \hbox{in}\, \, \RN.
\end{equation}
and proved that every positive solution of \eqref{plim} must assume the form
\begin{equation}\label{Uj}
     \Ud(x)=\lambda^{\frac{N-2}{2}}\mathcal{U}(\lambda(x-\xi))=\alpha_{N,\mu}\frac{\lambda^{\frac{N-2}{2}}}{\left(1+\lambda^2|x-\xi|^2\right)^{\frac{N-2}{2}}},\quad \lambda>0,\,\, x, \xi\in\mathbb R^N
 \end{equation}
 where
\begin{equation}\label{U}
    \mathcal{U}(x):=c_N\frac{1}{(1+|x|^2)^{\frac{N-2}{2}}},\quad c_N>0
\end{equation}
  and  $\alpha_{N,\mu}$ is a constant explicitly given and depends only on $N$ and $\mu$.\\

A central theme of this work is the role of the coupling parameter $\beta.$ When $\beta<0,$ the system \eqref{Pe} is in a competitive regime. In such cases, the interaction term acts as a repulsive force between the two components $u,v.$ \\
As suggested by the preliminary results (see \cite{CMP}, \cite{DMPP},\cite{DMPPP}) for $\beta<0$ and $|\beta|$ small, we will use the function $\mathcal{U}$ as approximation for the first component and for the second we will choose 
$$V:=\sum_{j=1}^k\Uj$$ where $k\geq 2$ is fixed, $\lambda>0$ and the points $\xi_j$  are arranged at the vertices of a regular polygon.\\
More precisely, 
\begin{equation}\label{xi}
\xi_j:=\rho\bigg(\cos\frac{2\pi(j-1)}{k},\sin\frac{2\pi(j-1)}{k},0,0\bigg),\quad \frac{1}{\lambda^2}+\rho^2=1,\quad j=1,...,k.
\end{equation}
The transition from the Choquard equation \eqref{plim} to the system \eqref{PO} is not trivial. While the Choquard equation and coupled Schr\"odinger systems with competitive interactions have been extensively investigated in isolation, their interaction remains an untouched frontier in the current literature. Our work is, to our knowledge, the first to bridge the gap between the theory of competitive systems and the critical Choquard equation.\\
We prove that even under the influence of nonlocal term, the system presents the same solution seen in \cite{CMP}. \\
Indeed, we have the following main theorem
\begin{theorem}\label{mainT}
    For every fixed $k\geq 2,$ there exists $\beta_0<0$ such that, for each $\beta\in [\beta_0,0),$ the system \eqref{Pe} has a solution $u=(u_1,u_2)$ with the form
    $$ u_1:=\mathcal{U}+\phi_1, \quad u_2:=\sum_{j=1}^k\Uj+\phi_2\quad \hbox{in}\, \, \RQ,$$ where $\mathcal{U}, \Uj$ defined in \eqref{U} and \eqref{Uj}, $\xi_j$ in \eqref{xi} and $\phi_1,\phi_2\in D^{1,2}(\RQ).$ \\
    Furthermore, 
    $$\|\phi_1\|_*, \|\phi_2\|_*\rightarrow0,\quad \frac{1}{\lambda}=O\bigg(e^{-c\frac{1}{|\beta|}}\bigg)$$ for some constant $c>0,$ as $|\beta|\rightarrow0.$
\end{theorem}
In this article, we employ a Lyapunov-Schmidt reduction method to prove the main result. For this reason, we only consider the case $N=4$, since, if $N\geq 5$ the coupling term $u_i^{\frac{2^*}{2}-1}u_j^{\frac{2^*}{2}}$ has sub-linear growth and the reduction process does not work. 
\section{The finite dimensional reduction}
We carry out the finite-dimensional reduction argument in a weighted space in order to obtain a good estimate for the error term. Let
\begin{equation}\label{normastar}
\|u\|_*=\sup_{x\in\RQ}\left(\sum_{j=1}^k\frac{1}{(1+\lambda|x-\xi_j|)^{1+\tau}}\right)^{-1}\lambda^{-1}|u(x)|\end{equation}
and
\begin{equation}\label{normastarstar}
\|h\|_{**}=\sup_{x\in\RQ}\left(\sum_{j=1}^k\frac{1}{(1+\lambda|x-\xi_j|)^{3+\tau}}\right)^{-1}\lambda^{-3}|h(x)|\end{equation} where $\tau<1$ small enough.\\
We set $$g(u):=\left(\int_\RQ \frac{|u(y)|^{\QM}}{|x-y|^\mu}\, dy\right)|u|^{\DM}u.$$ Moreover, we also set
$$g'(u)v=\bigg(\QM\bigg)\left(\int_\RN \frac{|u(y)|^{\DM}u(y)v(y)}{|x-y|^\mu}\, dy\right)|u|^{\DM}u+\bigg(\TM\bigg)\left(\int_\RN \frac{|u(y)|^{\QM}}{|x-y|^\mu}\, dy\right)|u|^{\DM}v.$$
Then we can rewrite the problem \eqref{Pe} in the following way 
\beq\label{Per}
\left\{\begin{aligned} &-\Delta u =g(u)+\beta uv^2\quad &\hbox{in}\,\, \RQ\\
&-\Delta v =g(v)+\beta u^2v \quad &\hbox{in}\,\, \RQ
\end{aligned}\right.\eeq
We look for a solution of the problem \eqref{Per} of the form
\begin{equation}\label{sol}
    u=\mathcal{U}+\phi_1,\quad v=\sum_{j=1}^k\Uj+\phi_2=:V+\phi_2
\end{equation}

where $k\geq 2$ is fixed, $\beta<0$, $\phi_1,\phi_2\in \mathcal{D}^{1,2}(\RQ)$ are small functions to be found, and $$\mathcal{U}(x):=c_4\frac{1}{(1+|x|^2)},\quad \Uj(x):=\alpha_{4,\mu}\frac{\lambda}{(1+\lambda^2|x-\xi_j|^2)}\quad  j=1,...,k$$ with 
$$\xi_j:=\rho\bigg(\cos\frac{2\pi(j-1)}{k},\sin\frac{2\pi(j-1)}{k},0,0\bigg),\quad \frac{1}{\lambda^2}+\rho^2=1,\quad j=1,...,k.$$
\subsection{Setting of the problem}
\begin{remark}
    We can observe that our problem \eqref{Pe} is invariant under Kelvin transform. Indeed, denoted by $$u_k(x):= \frac{1}{|x|^2}u(y)$$ and $$v_k(x):=\frac{1}{|x|^2}v(y)$$the Kelvin transform of $u$ and $v$ respectively with $y:=\frac{x}{|x|^2},$ it results that 
    $$-\Delta u_k(x)=\frac{1}{|x|^6}(-\Delta u)(y)=\frac{1}{|x|^6}\bigg[(|y|^{-\mu}*u^{\QM})u^{\TM}(y)+\beta u(y) v^2(y)\bigg].$$ Since,  $$u(y)=|x|^2u_k(x),\quad v^2(y)=|x|^4v_k^2(x),\quad  u(y)^{\TM}=|x|^{6-\mu}u_k(x)^{\TM}$$ and 
    \begin{equation*}
        \begin{aligned}
            (|y|^{\mu}*u^{\QM})&=\int_\RQ\frac{1}{|y-z|^\mu}u^{\QM}(z)\, dz=\int_\RQ\frac{|x|^\mu |w|^\mu}{|x-w|^\mu}|w|^{8-\mu}u_k^{\QM}(w)\frac{1}{|w|^8}\, dw\\
            &=|x|^{\mu}(|x|^{-\mu}*u_k^{\QM})
        \end{aligned}
    \end{equation*}
    using the following geometric identity $$\bigg|\frac{x}{|x|^2}-\frac{w}{|w|^2}\bigg|=\frac{|x-w|}{|x||w|}.$$ \\
    Therefore, 
    \begin{equation*}
        \begin{aligned}
            -\Delta u_k(x)&=\frac{1}{|x|^6}\bigg[|x|^\mu(|x|^{-\mu}*u_k^{\QM})|x|^{6-\mu}u_k^{\TM}(x)+\beta |x|^2u_k(x)|x|^4v_k^2(x)\bigg]\\
            &=(|x|^{-\mu}*u_k^{\QM})u_k^{\TM}(x)+\beta u_k(x)v_k^2(x).
        \end{aligned}
    \end{equation*}
    Similarly,
    \begin{equation*}
        -\Delta v_k(x)=(|x|^{-\mu}*v_k^{\QM})v_k^{\TM}(x)+\beta u_k^2(x)v_k(x).
    \end{equation*}
\end{remark}
Given a function $\phi\in \mathcal{D}^{1,2}(\RQ), $ let us consider the following invariances:
\begin{itemize}
    \item[(\textbf{A})] Evenness with respect to the $x_2,x_3,x_4$ variables, i.e.,
    $$\phi(x_1,x_2,x_3,x_4)=\phi(x_1,-x_2,x_3,x_4)=\phi(x_1,x_2,-x_3,x_4)=\phi(x_1,x_2,x_3,-x_4);$$
    \item[(\textbf{B})] Invariance under rotation of $\frac{2\pi}{k}$ in the $x_1,x_2$- variables, i.e.,
    $$\phi(\Theta_k(x_1,x_2),x_3,x_4)=\phi(x_1,x_2,x_3,x_4), $$ where $\Theta_k$ is the rotation matrix;
    \item[(\textbf{C})] Invariance under Kelvin transform, i.e., 
    $$\phi(x)=\frac{1}{|x|^2}\phi\bigg(\frac{x}{|x|^2}\bigg).$$
\end{itemize}
We define the associated space 
$$X:=\{\phi\in\mathcal{D}^{1,2}(\RQ):\phi \, \hbox{satisfies}\, (\textbf{A}), (\textbf{B}), (\textbf{C})\} $$ and we look for solutions $(u,v)$ in the space $X\times X.$\\
In order to perform the Lyapunov-Schmidt reduction it is very important to understand the kernel of the linearized operator associated to the limit problem \eqref{plim}. Very recently, in \cite{LI}, it has been shown that the solutions of \eqref{plim} are non-degenerate provide $0<\mu\leq 4,$ namely if we look for bounded solution of the linearized equation
\begin{equation}\label{lineq}
    -\Delta \mathcal{Z}=g'(\mathcal{U})\mathcal{Z}
\end{equation}
then $\mathcal{Z}\in \hbox{span}<Z^0,Z^1,...,Z^4>$ where 
$$Z^0(x):=\frac{1-|x|^2}{(1+|x|^2)^2} \quad \hbox{and}\quad Z^i(x):=\frac{x_i}{(1+|x|^2)^2},\, i=1,...,4.$$
Given $\lambda>0$ and $\xi_j$ given in \eqref{xi}, we can define 
$$Z^i_{\lambda,\xi_j}(x):=\lambda Z^i(\lambda(x-\xi_j)), \quad i=0,1,...,4,\quad  j=1,...,k$$ and 
\begin{equation}\label{Z}
    Z(x):=\sum_{j=1}^kZ^0_{\lambda,\xi_j}(x).
\end{equation}
We introduce the spaces 
$$K:=X\cap\hbox{span}\{Z\}, \quad K^\perp:=\{\phi\in X:<\phi,Z>=0\}$$
and the orthogonal projections
$$\Pi: X\times X\rightarrow X\times K \quad \hbox{and} \quad \Pi^\perp:X\times X\rightarrow X\times K^\perp.$$
If $(u,v)$ are given by \eqref{sol}, we can rewrite the system \eqref{Per} in terms of $(\phi_1,\phi_2)$ as 
\begin{equation}\label{pi}
    \Pi\{\mathcal{L}(\phi_1,\phi_2)-\mathcal{E}-\mathcal{N}(\phi_1,\phi_2)\}=0
\end{equation}
\begin{equation}\label{pio}
    \Pi^\perp\{\mathcal{L}(\phi_1,\phi_2)-\mathcal{E}-\mathcal{N}(\phi_1,\phi_2)\}=0
\end{equation}

where the linear operator $\mathcal{L}=(\mathcal{L}_1,\mathcal{L}_2)$ is defined by 
\beq\label{L}
\left\{\begin{aligned} &\mathcal{L}_1(\phi_1):=-\Delta\phi_1-g'(\mathcal{U})\phi_1\\
&\mathcal{L}_2(\phi_2):=-\Delta\phi_2-g'(V)\phi_2
\end{aligned}\right.\eeq
the error term $\mathcal{E}=(\mathcal{E}_1,\mathcal{E}_2)$ by 
\beq\label{E}
\left\{\begin{aligned} &\mathcal{E}_1:=\beta \mathcal{U}V^2\\
&\mathcal{E}_2:=g(V)-\sum_{j=1}^kg(\Uj)+\beta\mathcal{U}^2V
\end{aligned}\right.\eeq
and the quadratic term $\mathcal{N}=(\mathcal{N}_1,\mathcal{N}_2)$ by 
\beq\label{N}
\left\{\begin{aligned} &\mathcal{N}_1(\phi_1,\phi_2):=g(\mathcal{U}+\phi_1)-g(\mathcal{U})-g'(\mathcal{U})\phi_1+2\beta\mathcal{U}V\phi_2+\beta \mathcal{U}\phi_2^2+\beta V^2\phi_1+2\beta V\phi_1\phi_2+\beta\phi_1\phi_2^2\\
&\mathcal{N}_2(\phi_1,\phi_2):=g(V+\phi_2)-g(V)-g'(V)\phi_2+\beta \mathcal{U}^2\phi_2+2\beta \mathcal{U}V\phi_1+2\beta\mathcal{U}\phi_1\phi_2+\beta V\phi_1^2+\beta \phi_1^2\phi_2
\end{aligned}\right.\eeq

\subsection{The invertibility of linear operator $\mathcal{L}$}
We consider the following linear problem 
\beq\label{Plin}
\left\{\begin{aligned} &\mathcal{L}_1(\phi_1)=h_1\\&\mathcal{L}_2(\phi_2)=h_2+\mathfrak{c}\sum_{j=1}^kg'(\Uj)\Zj\\
&\phi_1,\phi_2\in X\\
&\sum_{j=1}^k\int_\RQ g'(\Uj)\Zj\phi_2\, dx =0
\end{aligned}\right.\eeq
for some real number $\mathfrak{c}.$\\

Reasoning as in Lemma 3.1 of \cite{GM} one can show the following result
\begin{lemma}
    Suppose that $\phi_{1,n},\phi_{2,n}$ solve \eqref{Plin} for $h_1=h_{1,n}$ and $h_2=h_{2,n}$.\\ If $\|h_{1,n}\|_{**}\rightarrow0$ and $\|h_{2,n}\|_{**}\rightarrow0$ as $n\rightarrow\infty,$ then $\|\phi_{1,n}\|_*\rightarrow 0$ and $\|\phi_{2,n}\|_*\rightarrow 0$.
\end{lemma}
\begin{proof}
    We argue by contradiction. Suppose that there exist $n\rightarrow\infty,\lambda_n$ such that $ \frac{1}{\lambda_n}\in (0,e^{-\frac{1}{\sqrt{|\beta|}}})$ and $\phi_{1,n},\phi_{2,n}$ solving \eqref{Plin} for $h_1=h_{1,n},h_2=h_{2,n}$, with $\|h_{1,n}\|_{**}\rightarrow 0, \|h_{2,n}\|_{**}\rightarrow 0$ and $\|\phi_{1,n}\|_*+\|\phi_{2,n}\|_*\geq c>0.$ \\
    We may assume that 
    \begin{equation}\label{contradiction}
        \|\phi_{1,n}\|_*+\|\phi_{2,n}\|_*=1.
    \end{equation}
    For simplicity, we let $\lambda:=\lambda_n.$\\
    \textbf{STEP 1:}\\
    Now, by the first equation, we have that 
    \begin{equation*}
        \begin{aligned}
            \|\phi_{1,n}\|^2_X&=\int_\RQ|\nabla\phi_{1,n}|^2\, dx =\int_\RQ g'(\mathcal{U})|\phi_{1,n}|^2\, dx +\int_\RQ h_{1,n}|\phi_{1,n}|\, dx 
        \end{aligned}
    \end{equation*}
    where 
    \begin{equation*}
        \begin{aligned}
            \int_\RQ h_{1,n}|\phi_{1,n}|\, dx&\lesssim \|h_{1,n}\|_{**}\|\phi_{1,n}\|_*\sum_{j=1}^k\int_\RQ\frac{\lambda^{4}}{(1+\lambda|x-\xi_j|)^{4+2\tau}}\, dx\\
            &\lesssim\|h_{1,n}\|_{**}\|\phi_{1,n}\|_*.
        \end{aligned}
    \end{equation*}
    Furthermore, chosen $R:=\frac{|x|}{2},$ we have that 
    \begin{equation*}
        \begin{aligned}
            &|x|^{-\mu}*\mathcal{U}^{\TM}\lesssim\int_{\RQ}\frac{1}{|x-y|^\mu}\frac{1}{(1+|y|^2)^{\TM}}\, dy\\
            &=\int_{B_R(x)}\frac{1}{|x-y|^\mu}\frac{1}{(1+|y|^2)^{\TM}}\, dy+\int_{\RQ\setminus B_R(x)}\frac{1}{|x-y|^\mu}\frac{1}{(1+|y|^2)^{\TM}}\, dy\\
            &\lesssim \frac{1}{(1+|x|^2)^{3-\frac{\mu}{2}}}\int_{B_R(0)}\frac{1}{|z|^\mu}\, dz+\int_{\RQ\setminus B_R(x)}\frac{1}{|y|^{6}}\, dy\\
            &\lesssim \frac{1}{(1+|x|^2)}
        \end{aligned}
    \end{equation*}
    and by Remark \ref{R}
    $$|x|^{-\mu}*\mathcal{U}\lesssim\frac{1}{(1+|x|^2)^{\frac{\mu}{2}}}$$
    it results that 
    \begin{equation*}
        \begin{aligned}
            &\int_\RQ g'(\mathcal{U})|\phi_{1,n}|^2\, dx \lesssim \int_\RQ(|x|^{-\mu}*\mathcal{U}^{\TM}|\phi_{1,n}|^2)\mathcal{U}^{\TM}\, dx +\int_\RQ(|x|^{-\mu}*\mathcal{U}^{\QM})\mathcal{U}^{\DM}|\phi_{1,n}|^2\, dx \\
            &\lesssim \int_\RQ(|x|^{-\mu}*\mathcal{U}^{\TM})
            \mathcal{U}^{\TM}|\phi_{1,n}|^2\, dx +\int_\RQ(|x|^{-\mu}*\mathcal{U}^{\QM})\mathcal{U}^{\DM}|\phi_{1,n}|^2\, dx \\
            &\lesssim \|\phi_{1,n}\|_*^2\int_\RQ\frac{1}{(1+|x|^2)^{4-\frac{\mu}{2}}}\sum_{j=1}^k\frac{\lambda^2}{(1+\lambda|x-\xi_j|)^{2+2\tau}}\, dx\\
            &+\|\phi_{1,n}\|^2_*\int_\RQ\frac{1}{(1+|x|^2)^{2}}\sum_{j=1}^k\frac{\lambda^2}{(1+\lambda|x-\xi_j|)^{2+2\tau}}\, dx \\
            &\lesssim\frac{\|\phi_{1,n}\|_*^2}{\lambda^{2\tau}}+\frac{\|\phi_{1,n}\|_*^2}{\lambda^{2\tau}}\lesssim \frac{\|\phi_{1,n}\|_*^2}{\lambda^{2\tau}}.
        \end{aligned}
    \end{equation*}
    Therefore, since $\|\phi_{1,n}\|_*\leq1$ by \eqref{contradiction} and $\|h_{1,n}\|_{**}\rightarrow0$, we have that
    $$\|\phi_{1,n}\|_X\lesssim \frac{\|\phi_{1,n}\|_*^2}{\lambda^{2\tau}} +\|h_{1,n}\|_{**}\|\phi_{1,n}\|_*\rightarrow0.$$
    \textbf{STEP 2:}\\
    By the second equation of \eqref{Plin}, we have
    \begin{equation*}
        \begin{aligned}
            |\phi_{2,n}(x)|&\lesssim\underbrace{\int_\RQ\frac{1}{|x-y|^2}(|y|^{-\mu}*V^{\TM}|\phi_{2,n}|)V^{\TM}(y)\, dy}_{(I)} +\\
            &+\underbrace{\int_\RQ\frac{1}{|x-y|^2}(|y|^{-\mu}*V^{\QM})V^{\DM}(y)|\phi_{2,n}(y)|\, dy}_{(II)} +\\
            &+|\mathfrak{c}|\bigg[\bigg|\sum_{j=1}^k\underbrace{\int_\RQ\frac{1}{|y-x|^2}(|y|^{-\mu}*\Uj^{\TM}\Zj)\Uj^{\TM}(y)\, dy}_{(III)} \bigg|+\\
            &+\bigg|\sum_{j=1}^k\underbrace{\int_\RQ \frac{1}{|y-x|^2}(|y|^{-\mu}*\Uj^{\QM})\Uj^{\DM}(y)\Zj(y)\, dy}_{(IV)} \bigg|\bigg]+\\
            &+\underbrace{\int_\RQ \frac{1}{|y-x|^2}|h_{2,n}(y)|\, dy}_{(V)}.
        \end{aligned}
    \end{equation*}
    By using Lemma \ref{LEM2} and Remark \ref{REM}, it results that
    \begin{equation}\label{b}
        \begin{aligned}
            (|y|^{-\mu}*V^{\TM}|\phi_{2,n}|)&\lesssim\int_\RQ\frac{1}{|y-x|^\mu}\sum_{j=1}^k\frac{\lambda^{\TM}}{(1+\lambda|x-\xi_j|)^{6-\mu}}|\phi_{2,n}(x)|\, dx \\
            &\lesssim\|\phi_{2,n}\|_*\int_\RQ\frac{1}{|y-x|^\mu}\sum_{j=1}^k\frac{\lambda^{\QM}}{(1+\lambda|x-\xi_j|)^{7-\mu+\tau}}\, dx\\
            &\lesssim \|\phi_{2,n}\|_*\sum_{j=1}^k\frac{\lambda^{\frac{\mu}{2}}}{(1+\lambda|y-\xi_j|)^{3+\tau-\eta}}
        \end{aligned}
    \end{equation}
    with $\eta>0$ small enough and 
    $$(|y|^{-\mu}*V^{\QM})\lesssim\sum_{j=1}^k\frac{\lambda^{\frac{\mu}{2}}}{(1+\lambda|y-\xi_j|)^\mu}$$ 
   and so
    \begin{equation*}
        \begin{aligned}
            (I)&\lesssim \|\phi_{2,n}\|_*\int_\RQ \frac{1}{|y-x|^2}\sum_{j=1}^k \frac{\lambda^{\frac{\mu}{2}}}{(1+\lambda|y-\xi_j|)^{3+\tau-\eta}}\cdot \frac{\lambda^{\TM}}{(1+\lambda|y-\xi_j|)^{6-\mu}}\, dy\\
            &\lesssim \|\phi_{2,n}\|_*\int_\RQ \frac{1}{|y-x|^2}\sum_{j=1}^k \frac{\lambda^3}{(1+\lambda|y-\xi_j|)^{9+\tau-\mu-\eta}}\, dy \\
            &\lesssim\|\phi_{2,n}\|_*\lambda\int_\RQ\sum_{j=1}^k\frac{1}{|y-\lambda(x-\xi_j)|^2}\frac{1}{(1+|y|)^{9+\tau-\mu-\eta}}\, dy \\
            &\lesssim \|\phi_{2,n}\|_*\lambda\sum_{j=1}^k\frac{1}{(1+\lambda|x-\xi_j|)^{7+\tau-\mu-\eta}}\lesssim \|\phi_{2,n}\|_*\lambda\sum_{j=1}^k\frac{1}{(1+\lambda|x-\xi_j|)^{1+\tau+\theta_1}}
        \end{aligned}
    \end{equation*}
    with $\theta_1:=6-\mu-\eta$ and 
    \begin{equation*}
        \begin{aligned}
            (II)&\lesssim\int_\RQ\frac{1}{|y-x|^2}\sum_{j=1}^k\frac{\lambda^{\frac{\mu}{2}}}{(1+\lambda|y-\xi_j|)^\mu}\frac{\lambda^{\DM}}{(1+\lambda|y-\xi_j|)^{4-\mu}}|\phi_{2,n}(y)|\, dy\\
            &\lesssim\|\phi_{2,n}\|_*\int_\RQ\frac{1}{|y-x|^2}\sum_{j=1}^k\frac{\lambda^3}{(1+\lambda|y-\xi_j|)^{5+\tau}}\, dy \\
            &\lesssim\|\phi_{2,n}\|_*\lambda\int_\RQ\frac{1}{|y-\lambda(x-\xi_j)|^2}\frac{1}{(1+|y|)^{5+\tau}}\, dy \\
            &\lesssim\|\phi_{2,n}\|_*\lambda \sum_{j=1}^k\frac{1}{(1+\lambda|x-\xi_j|)^{3+\tau}}\lesssim\|\phi_{2,n}\|_*\lambda \sum_{j=1}^k\frac{1}{(1+\lambda|x-\xi_j|)^{1+\tau+\theta_2}}
        \end{aligned}
    \end{equation*}
    with $\theta_2:=2.$
    In addition, by Lemma \ref{LEM2},
    \begin{equation}\label{a}
        \begin{aligned}
            (|y|^{-\mu}*\Uj^{\TM}\Zj)&\lesssim \int_\RQ\frac{1}{|y-x|^\mu}\sum_{j=1}^k \frac{\lambda^{\TM}}{(1+\lambda|x-\xi_j|)^{6-\mu}}\frac{\lambda}{(1+\lambda|x-\xi_j|)^2}\, dx\\
            &\lesssim\int_\RQ\frac{1}{|y-x|^\mu}\sum_{j=1}^k \frac{\lambda^{\QM}}{(1+\lambda|x-\xi_j|)^{8-\mu}}\\
            &\lesssim \sum_{j=1}^k\frac{\lambda^{\frac{\mu}{2}}}{(1+\lambda|y-\xi_j|)^{4-\eta}}
        \end{aligned}
    \end{equation}
    then, 
    \begin{equation*}
        \begin{aligned}
            (III)&\lesssim\int_\RQ\frac{1}{|y-x|^2}\sum_{j=1}^k\frac{\lambda^{\frac{\mu}{2}}}{(1+\lambda|y-\xi_j|)^{4-\eta}}\frac{\lambda^{\TM}}{(1+\lambda|y-\xi_j|)^{6-\mu}}\, dy\\
            &\lesssim\int_\RQ\frac{1}{|y-x|^2}\sum_{j=1}^k\frac{\lambda^3}{(1+\lambda|y-\xi_j|)^{10-\mu-\eta}}\\
            &\lesssim\lambda\int_\RQ\sum_{j=1}^k\frac{1}{|y-\lambda(x-\xi_j)|^2}\frac{1}{(1+|y|)^{10-\mu-\eta}}\, dy\\
            &\lesssim\sum_{j=1}^k\frac{\lambda}{(1+\lambda|x-\xi_j|)^{8-\mu-\eta}}\lesssim\sum_{j=1}^k\frac{\lambda}{(1+\lambda|x-\xi_j|)^{1+\tau}}
        \end{aligned}
    \end{equation*}
    and, by Remark \ref{REM},
    \begin{equation*}
        \begin{aligned}
            (IV)&\lesssim\int_\RQ\frac{1}{|y-x|^2}\sum_{j=1}^k\frac{\lambda^{\frac{\mu}{2}}}{(1+\lambda|y-\xi_j|)^\mu}\frac{\lambda^{\DM}}{(1+\lambda|y-\xi_j|)^{4-\mu}}\frac{\lambda}{(1+\lambda|y-\xi_j|)^2}\, dy \\
            &\lesssim \int_\RQ\frac{1}{|y-x|^2}\sum_{j=1}^k\frac{\lambda^3}{(1+\lambda|y-\xi_j|)^6}\, dy \\
            &\lesssim\lambda\int_\RQ\sum_{j=1}^k\frac{1}{|y-\lambda(x-\xi_j)|^2}\frac{1}{(1+|y|)^6}\, dy \\
            &\lesssim\sum_{j=1}^k\frac{\lambda}{(1+\lambda|x-\xi_j|)^{4}}
            \lesssim\sum_{j=1}^k\frac{\lambda}{(1+\lambda|x-\xi_j|)^{1+\tau}}.
        \end{aligned}
    \end{equation*}
    Finally, 
    \begin{equation*}
        \begin{aligned}
            (V)&\lesssim\|h_{2,n}\|_{**}\int_\RQ\frac{1}{|y-x|^2}\sum_{j=1}^k\frac{\lambda^3}{(1+\lambda|y-\xi_j|)^{3+\tau}}\, dy \\
            &\lesssim \|h_{2,n}\|_{**}\lambda\int_\RQ\sum_{j=1}^k\frac{1}{|y-\lambda(x-\xi_j)|^2}\frac{1}{(1+|y|)^{3+\tau}}\, dy \\
            &\lesssim \|h_{2,n}\|_{**}\lambda\sum_{j=1}^k\frac{1}{(1+\lambda|x-\xi_j|)^{1+\tau}}.
        \end{aligned}
    \end{equation*}
    In the following, we are going to estimate $\mathfrak{c}.$\\
    Multiplying the second equation of \eqref{Plin} by $\Zo$ and integrating, we see that $\mathfrak{c}$ satisfies:
    \begin{equation}\label{stima}
        \begin{aligned}
            &\sum_{j=1}^k<g'(\Uj)\Zj,\Zo>\mathfrak{c}=<-\Delta\phi_{2,n}-g'(V)\phi_{2,n},\Zo>-<h_{2,n},\Zo>.
        \end{aligned}
    \end{equation}
    Now, 
    \begin{equation}\label{a'}
        \begin{aligned}
            &|<h_{2,n}, \Zo>|\lesssim \int_\RQ |h_{2,n}(x)|\frac{\lambda}{(1+\lambda|x-\xi_1|)^2}\, dx \\
            &\lesssim \|h_{2,n}\|_{**}\int_\RQ \frac{\lambda}{(1+\lambda|x-\xi_1|)^2}\sum_{j=1}^k \frac{\lambda^3}{(1+\lambda|x-\xi_j|)^{3+\tau}}\, dx \\
            &\lesssim\|h_{2,n}\|_{**}\bigg(\int_\RQ\frac{\lambda^4}{(1+\lambda|x-\xi_1|)^{5+\tau}}\, dx +\sum_{j=2}^k\int_\RQ\frac{\lambda^4}{(1+\lambda|x-\xi_1|)^2}\frac{1}{(1+\lambda|x-\xi_j|)^{3+\tau}}\, dx \bigg)\\
            &\lesssim\|h_{2,n}\|_{**}\bigg(\int_\RQ\frac{1}{(1+|y|)^{5+\tau}}\, dy+\frac{1}{\lambda^\alpha}\int_\RQ \frac{1}{(1+|y|)^{5+\tau-\alpha}}\, dy\bigg)\\
            &\lesssim \|h_{2,n}\|_{**}
        \end{aligned}
    \end{equation}
    with $\alpha<2.$ Moreover, by \eqref{a}
    \begin{equation*}
        \begin{aligned}
            &\int_\RQ(|x|^{-\mu}*\Uj^{\TM}\Zj)\Uj^{\TM}\Zo\, dx\\ &\lesssim\int_\RQ\sum_{j=1}^k\frac{\lambda^{\frac{\mu}{2}}}{(1+\lambda|x-\xi_j|)^{4-\eta}}\frac{\lambda^{\TM}}{(1+\lambda|x-\xi_j|)^{6-\mu}}\frac{\lambda}{(1+\lambda|x-\xi_1|)^{2}}\, dx \\
            &\lesssim\int_\RQ\sum_{j=1}^k\frac{\lambda^4}{(1+\lambda|x-\xi_j|)^{10-\mu-\eta}}\frac{1}{(1+\lambda|x-\xi_1|)^{2}}\, dx \\
            &\lesssim \int_\RQ \frac{\lambda^4}{(1+\lambda|x-\xi_1|)^{12-\mu-\eta}}\, dx +\sum_{j=2}^k\int_\RQ\frac{\lambda^4}{(1+\lambda|x-\xi_j|)^{10-\mu-\eta}}\frac{1}{(1+\lambda|x-\xi_1|)^{2}}\, dx \\
            &\lesssim\int_\RQ\frac{1}{(1+|y|)^{12-\mu-\eta}}\, dy +\frac{1}{\lambda^2}\int_\RQ\frac{1}{(1+|y|)^{10-\mu-\eta}}\, dy\leq C
        \end{aligned}
    \end{equation*}
    and by Remark \ref{REM}
    \begin{equation*}
        \begin{aligned}
            &\int_\RQ(|x|^{-\mu}*\Uj^{\QM})\Uj^{\DM}\Zj\Zo\, dx\\
            &\lesssim\int_\RQ\sum_{j=1}^k\frac{\lambda^{\frac{\mu}{2}}}{(1+\lambda|x-\xi_j|)^\mu}\frac{\lambda^{\DM}}{(1+\lambda|x-\xi_j|)^{4-\mu}}\frac{\lambda}{(1+\lambda|x-\xi_j|)^2}\frac{\lambda}{(1+\lambda|x-\xi_1|)^2}\, dx \\
            &\lesssim\int_\RQ\sum_{j=1}^k \frac{\lambda^4}{(1+\lambda|x-\xi_j|)^6}\frac{1}{(1+\lambda|x-\xi_1|)^2}\, dx \\
            &\lesssim \int_\RQ \frac{\lambda^4}{(1+\lambda|x-\xi_1|)^8}\, dx+\sum_{j=2}^k\int_\RQ \frac{\lambda^4}{(1+\lambda|x-\xi_j|)^6}\frac{1}{(1+\lambda|x-\xi_1|)^2}\, dx \\
            &\lesssim\int_\RQ\frac{1}{(1+|y|)^8}\, dy+\frac{1}{\lambda^2}\int_\RQ\frac{1}{(1+|y|)^6}\, dy\leq \Tilde{C}.
        \end{aligned}
    \end{equation*}
    Therefore, 
    \begin{equation}\label{left}
        \bigg|\sum_{j=1}^k<g'(\Uj)\Zj,\Zo>\bigg|\leq \bar C.
    \end{equation}
    To estimate the right-hand side of \eqref{stima}, we can observe that by \eqref{b}
    \begin{equation*}
        \begin{aligned}
            &\bigg|\int_\RQ(|x|^{-\mu}*V^{\TM}\phi_{2,n})V^{\TM}\Zo\, dx \bigg|\\
            &\lesssim\|\phi_{2,n}\|_*\int_\RQ\sum_{j=1}^k\frac{\lambda^{\frac{\mu}{2}}}{(1+\lambda|x-\xi_j|)^{3+\tau-\eta}}\frac{\lambda^{\TM}}{(1+\lambda|x-\xi_j|)^{6-\mu}}\frac{\lambda}{(1+\lambda|x-\xi_1|)^2}\, dx \\
            &\lesssim\|\phi_{2,n}\|_*\int_\RQ\sum_{j=1}^k\frac{\lambda^4}{(1+\lambda|x-\xi_j|)^{9+\tau-\mu-\eta}}\frac{1}{(1+\lambda|x-\xi_1|)^2}\, dx \\
            &\lesssim\|\phi_{2,n}\|_*\bigg(\int_\RQ\frac{\lambda^4}{(1+\lambda|x-\xi_1|)^{11+\tau-\mu-\eta}}\, dx +\sum_{j=2}^k\int_\RQ\frac{\lambda^4}{(1+\lambda|x-\xi_j|)^{9+\tau-\mu-\eta}}\frac{1}{(1+\lambda|x-\xi_1|)^2}\, dx \bigg)\\
            &\lesssim\|\phi_{2,n}\|_*\bigg(\int_\RQ\frac{1}{(1+|y|)^{11+\tau-\mu-\eta}}\, dy+\frac{1}{\lambda^2}\int_\RQ\frac{1}{(1+|y|)^{9+\tau-\mu-\eta}}\, dy\bigg)\lesssim \|\phi_{2,n}\|_*
        \end{aligned}
    \end{equation*}
    and by Remark \ref{REM}
    \begin{equation*}
        \begin{aligned}
            &\bigg|\int_\RQ(|x|^{-\mu}*V^{\QM})V^{\DM}\phi_{2,n}\Zo\, dx \bigg|\\
            &\lesssim\int_\RQ\sum_{j=1}^k\frac{\lambda^{\frac{\mu}{2}}}{(1+\lambda|x-\xi_j|)^\mu}\frac{\lambda^{\DM}}{(1+\lambda|x-\xi_j|)^{4-\mu}}|\phi_{2,n}(x)|\frac{\lambda}{(1+\lambda|x-\xi_1|)^2}\, dx \\
            &\lesssim\|\phi_{2,n}\|_*\int_\RQ\sum_{j=1}^k\frac{\lambda^4}{(1+\lambda|x-\xi_j|)^{5+\tau}}\frac{1}{(1+\lambda|x-\xi_1|)^2}\, dx \\
            &\lesssim\|\phi_{2,n}\|_*\bigg(\int_\RQ\frac{\lambda^4}{(1+\lambda|x-\xi_1|)^{7+\tau}}\, dx +\sum_{j=2}^k\int_\RQ \frac{\lambda^4}{(1+\lambda|x-\xi_j|)^{5+\tau}}\frac{1}{(1+\lambda|x-\xi_1|)^2}\, dx \bigg)\\
            &\lesssim\|\phi_{2,n}\|_*\bigg(\int_\RQ\frac{1}{(1+|y|)^{7+\tau}}\, dy+\frac{1}{\lambda^2}\int_\RQ\frac{1}{(1+|y|)^{5+\tau}}\, dy\bigg)\lesssim\|\phi_{2,n}\|_*.
        \end{aligned}
    \end{equation*}
    So, 
    \begin{equation}\label{right}
        |<g'(V)\phi_{2,n},\Zo>|\lesssim\|\phi_{2,n}\|_*.
    \end{equation}
    Finally,
    \begin{equation*}
        \begin{aligned}
            &\int_\RQ-\Delta\Zo\phi_{2,n}\, dx=\int_\RQ g'(\U)\phi_{2,n}\Zo\, dx\\
            &\lesssim\int_\RQ(|x|^{-\mu}*\U^{\TM}\phi_{2,n})\U^{\TM}\Zo\, dx +\int_\RQ(|x|^{-\mu}*\U^{\QM})\U^{\DM}\phi_{2,n}\Zo\, dx.
        \end{aligned}
    \end{equation*}
    Since
    \begin{equation*}
        \begin{aligned}
            |x|^{-\mu}*\U^{\TM}|\phi_{2,n}|&\lesssim\|\phi_{2,n}\|_*\int_\RQ\frac{1}{|x-y|^\mu}\frac{\lambda^{\TM}}{(1+\lambda|y-\xi_1|)^{6-\mu}}\sum_{j=1}^k\frac{\lambda}{(1+\lambda|y-\xi_j|)^{1+\tau}}\, dy\\
            &\lesssim\|\phi_{2,n}\|_*\bigg(\int_\RQ\frac{1}{|x-y|^\mu}\frac{\lambda^{\QM}}{(1+\lambda|y-\xi_1|)^{7-\mu+\tau}}\, dy +\\
            &+\sum_{j=2}^k\int_\RQ\frac{1}{|x-y|^\mu}\frac{\lambda^{\QM}}{(1+\lambda|y-\xi_1|)^{6-\mu}}\frac{1}{(1+\lambda|y-\xi_j|)^{1+\tau}}\, dy\bigg)\\
            &\lesssim\|\phi_{2,n}\|_*\bigg(\frac{\lambda^{\frac{\mu}{2}}}{(1+\lambda|x-\xi_1|)^{3+\tau-\eta}}+\frac{1}{\lambda^{1+\tau}}\frac{\lambda^{\frac{\mu}{2}}}{(1+\lambda|x-\xi_1|)^{2-\eta}}\bigg)\\
            &\lesssim\|\phi_{2,n}\|_*\frac{\lambda^{\frac{\mu}{2}}}{(1+\lambda|x-\xi_1|)^{3+\tau-\eta}}
        \end{aligned}
    \end{equation*}
    it results that 
    \begin{equation}\label{c}
        \begin{aligned}
           & \int_\RQ-\Delta\Zo\phi_{2,n}\, dx\lesssim\\
           &\|\phi_{2,n}\|_*\bigg(\int_\RQ\frac{\lambda^{\frac{\mu}{2}}}{(1+\lambda^2|x-\xi_1|^2)^{\frac{3+\tau-\eta}{2}}}\frac{\lambda^{\TM}}{(1+\lambda^2|x-\xi_1|^2)^{\TM}}\lambda\frac{1-\lambda^2|x-\xi_1|^2}{(1+\lambda^2|x-\xi_1|^2)^2}\, dx +\\
            &+\int_\RQ\frac{\lambda^{\frac{\mu}{2}}}{(1+\lambda^2|x-\xi_1|^2)^{\frac{\mu}{2}}}\frac{\lambda^{\DM}}{(1+\lambda^2|x-\xi_1|^2)^{\DM}}\sum_{j=1}^k\frac{\lambda}{(1+\lambda|x-\xi_j|)^{1+\tau}}\lambda\frac{1-\lambda^2|x-\xi_1|^2}{(1+\lambda^2|x-\xi_1|^2)^2}\, dx \bigg)\\
            &\lesssim\|\phi_{2,n}\|_*\bigg(\int_\RQ\lambda^4\frac{1-\lambda^2|x-\xi_1|^2}{(1+\lambda^2|x-\xi_1|^2)^{\frac{13-\mu+\tau-\eta}{2}}}\, dx +\int_\RQ\lambda^4\frac{1-\lambda^2|x-\xi_1|^2}{(1+\lambda^2|x-\xi_1|^2)^{4+\frac{1+\tau}{2}}}\, dx+\\
            &+\sum_{j=2}^k\int_\RQ \frac{\lambda^4}{(1+\lambda^2|x-\xi_1|^2)^{4}}\frac{1}{(1+\lambda|x-\xi_j|)^{1+\tau}}\, dx\bigg)\\
            &\lesssim \|\phi_{2,n}\|_*\bigg(\frac{1}{\lambda^{7-\mu+\tau-\eta}}+\frac{1}{\lambda^{3+\tau}}+\frac{1}{\lambda^{1+\tau}}\int_\RQ\frac{\lambda^4}{(1+\lambda|x-\xi_1|)^{8}}\, dx \bigg)\\
            &\lesssim\frac{\|\phi_{2,n}\|_*}{\lambda^{1+\tau}}\lesssim\frac{\|\phi_{2,n}\|_*}{\lambda}.
        \end{aligned}
    \end{equation}
    Putting together \eqref{left}, \eqref{right}, \eqref{c} and \eqref{a'}, we have that 
    \begin{equation}
        \mathfrak{c}=\frac{1}{\lambda}(o\|\phi_{2,n}\|_*)+O(\|h_{2,n}\|_{**}).
    \end{equation}
    Thus, with $\theta:=\min\{\theta_1,\theta_2\}$
    \begin{equation}
        \|\phi_{2,n}\|_*\leq o(1)+\|h_{2,n}\|_{**}+\frac{\sum_{j=1}^k\frac{1}{(1+\lambda|x-\xi_j|)^{1+\tau+\theta}}}{\sum_{j=1}^k\frac{1}{(1+\lambda|x-\xi_j|)^{1+\tau}}}.
    \end{equation}
    To conclude the proof, we just apply the same argument of the Lemma 3.1 of \cite{GM}. 
\end{proof}
As in Lemma 3.2 of \cite{GM}, we can deduce the following result
\begin{proposition}\label{inveritibilità}
    There exist a constant $C>0$ such that for all $h=(h_1, h_2)\in L^\infty(\RQ)\times L^\infty(\RQ),$ problem \eqref{Plin} has an unique solution $\phi=(\phi_1,\phi_2).$ Besides,
    $$\|\phi\|_*\leq C\|h\|_{**}\quad\hbox{and
    }\quad  |\mathfrak{c}|\leq\frac{C}{\lambda}\|h_2\|_{**}.$$
\end{proposition}

In order to use contraction mapping theorem to prove that \eqref{Plin} is uniquely solvable in the set where $\|\phi_1\|_*$ and $\|\phi_2\|_*$ are small, we need to estimate $\mathcal{E}$ and $\mathcal{N}.$
\subsection{The size of the error term}
Let 
\begin{equation}\label{simmetria}
    \Omega_j:=\{x=(x',x'')\in \mathbb{R}^2\times\mathbb{R}^2:\bigg<\frac{x'}{|x'|},\frac{\xi_j'}{|\xi_j'|}\bigg>\geq \cos\frac{\pi}{k}\},\quad j=1,...,k
\end{equation}
\begin{proposition}\label{errore}
    There exists $\beta_0<0$ such that for any $\beta\in [\beta_0,0)$ and $\lambda$ such that $\frac{1}{\lambda}\in (0,e^{-\frac{1}{\sqrt{|\beta|}}})$, it holds 
    $$\|\mathcal{E}_1\|_{**}+\|\mathcal{E}_2\|_{**}=O\bigg(\frac{|\beta|}{\lambda}\bigg).$$
\end{proposition}
\begin{proof}
    Let's start by estimating $\mathcal{E}_1.$ Since
    $$|\mathcal{E}_1|\lesssim |\beta| \frac{1}{(1+|x|^2)}\sum_{j=1}^k\frac{\lambda^2}{(1+\lambda^2|x-\xi_j|^2)^2}\lesssim |\beta|\sum_{j=1}^k\frac{\lambda^2}{(1+\lambda|x-\xi_j|)^4}$$
    therefore
    $$\|\mathcal{E}_1\|_{**}\lesssim \frac{|\beta|}{\lambda}\sum_{j=1}^k\frac{1}{(1+\lambda|x-\xi_j|)^{1-\tau}}\lesssim \frac{|\beta|}{\lambda}.$$ Now, let us estimate $\|\mathcal{E}_2\|_{**}$.\\
    Using the assumed symmetry \eqref{simmetria}, we just need to estimate $\beta\mathcal{U}^2V$ in $\Omega_1$. Let $S:=\Omega_1\cap B_{\frac{1}{\sqrt{\lambda}}}(\xi_1).$ Then, for all $x\in S$, we have that 
    $$|\beta \mathcal{U}^2V|\lesssim |\beta|\frac{1}{(1+|x|^2)^2}\frac{\lambda}{(1+\lambda^2|x-\xi_1|^2)}\lesssim|\beta|\frac{\lambda}{(1+\lambda|x-\xi_1|)^2}$$ and so 
    \begin{equation}
        \begin{aligned}\label{B}
            \|\beta \mathcal{U}^2V\|_{**}\lesssim \frac{|\beta|}{\lambda^2}(1+\lambda|x-\xi_1|)^{1+\tau}\lesssim\frac{|\beta|}{\lambda^{2}}+\frac{|\beta|}{\lambda^{2}}\lambda^{\frac{1}{2}(1+\tau)}\lesssim\frac{|\beta|}{\lambda^{\frac{3}{2}-\frac{\tau}{2}}},\quad x\in S.
        \end{aligned}
    \end{equation}
    On the other hand, for all $x\in \Omega_1\setminus S,$ we have that
    \begin{equation*}
        \begin{aligned}
            |\beta \mathcal{U}^2V|&\lesssim |\beta|\frac{1}{(1+|x|^2)^2}\frac{\lambda}{(1+\lambda^2|x-\xi_1|^2)}\\
            &\lesssim|\beta|\bigg[\frac{1}{\lambda|x-\xi_1|^2}+\frac{1}{\lambda|x-\xi_1|^{2-\tau}}\sum_{j=2}^k\frac{1}{|\xi_j-\xi_1|^\tau}\bigg]\\
            &\lesssim |\beta|\frac{1}{\lambda|x-\xi_1|^{2-\tau}}
        \end{aligned}
    \end{equation*}
    Now, we determine $r>0,$ such that
    $$\frac{1}{\lambda|x-\xi_1|^{2-\tau}}\leq \lambda^{-r}\frac{\lambda^3}{(1+\lambda|x-\xi_1|)^{3+\tau}},\quad x\in \Omega_1\setminus S$$
    which is equivalent to
    $$\lambda^{1-\tau-r}\geq |x-\xi_1|^{1+2\tau}.$$
    Since $|x-\xi_1|\geq \lambda^{-\frac{1}{2}}$, it results that
    $$\lambda^{1-\tau-r}\geq\lambda^{-\frac{1}{2}(1+2\tau)}\iff r\leq \frac{3}{2}.$$
    Therefore, 
    \begin{equation}\label{BB}
        \|\beta \mathcal{U}^2 V\|_{**}\lesssim\frac{|\beta|}{\lambda^{\frac{3}{2}}}, \quad x\in \Omega_1\setminus S.
    \end{equation}
    Putting together \eqref{B} and \eqref{BB}, we have that 
    $$\|\beta\mathcal{U}^2V\|_{**}\lesssim\frac{|\beta|}{   \lambda^{\frac{3}{2}-\frac{\tau}{2}}}.$$
    Finally, 
    \begin{equation*}
        \begin{aligned}
            &|g(V)-\sum_{j=1}^kg(\Uj)|=\\
            &=\bigg[|x|^{-\mu}*\bigg(\sum_{j=1}^k\Uj\bigg)^{\QM}\bigg]\cdot \bigg(\sum_{j=1}^k\Uj\bigg)^{\TM}-\sum_{j=1}^k(|x|^{-\mu}*(\Uj)^{\QM})(\Uj)^{\TM}=\\
            &=\underbrace{\bigg[|x|^{-\mu}*\bigg(\bigg(\sum_{j=1}^k\Uj\bigg)^{\QM}-\sum_{j=1}^k\Uj^{\QM}\bigg)\bigg]\cdot \bigg(\sum_{j=1}^k\Uj\bigg)^{\TM}}_{(A)}+\\
            &+O\bigg(\underbrace{\sum_{i\not=j}^k(|x|^{-\mu}*(\Uj)^{\QM})(\Ui)^{\TM}}_{(B)}\bigg).
        \end{aligned}
    \end{equation*}
    We have that for any $x\in \Omega_j$ 
    $$\bigg|\bigg(\sum_{j=1}^k\Uj\bigg)^{\QM}-\sum_{j=1}^k\Uj^{\QM}\bigg|\lesssim(\Uj)^{\TM}\sum_{i\not=j}\Ui+\sum_{i\not=j}(\Ui)^{\QM}.$$ Now, using Lemma \ref{LEM1} and by taking $0<\alpha<2,$ we get that for any $x\in \Omega_j$ and $i\not=j$
    \begin{equation*}
        \begin{aligned}
            (\Uj)^{\TM}\Ui&\lesssim\frac{\lambda^{\TM}}{(1+\lambda|x-\xi_j|)^{6-\mu}}\frac{\lambda}{(1+\lambda|x-\xi_i|)^2}\\
            &\lesssim\frac{1}{\lambda^\alpha|\xi_i-\xi_j|^\alpha}\frac{\lambda^{\QM}}{(1+\lambda|x-\xi_j|)^{8-\mu-\alpha}}
        \end{aligned}
    \end{equation*}
    and hence 
    $$(\Uj)^{\TM}\sum_{i\not=j}\Ui\lesssim\frac{1}{\lambda^\alpha}\frac{\lambda^{\QM}}{(1+\lambda|x-\xi_j|)^{8-\mu-\alpha}}.$$ Now, by Lemma \ref{LEM2} with $\theta:=\QM$ and $\gamma:=4-\alpha+(4-\mu)$
    \begin{equation*}
        \begin{aligned}
            |x|^{-\mu}*(\Uj)^{\TM}\sum_{i\not=j}\Ui&\lesssim\frac{1}{\lambda^\alpha}\lambda^{\frac{\mu}{2}}\int_\RQ\frac{1}{|y|^\mu}\frac{1}{(1+|\lambda(x-\xi_j)-y|)^{8-\mu-\alpha}}\\
            &\lesssim \frac{1}{\lambda^\alpha} \frac{\lambda^{\frac{\mu}{2}}}{(1+\lambda|x-\xi_j|)^{4-\alpha-\eta}}
        \end{aligned}
    \end{equation*}
    for $\eta>0$ small. Moreover by using Remark \ref{REM} we have that 
    $$|x|^{-\mu}*\sum_{i\not=j}(\Ui)^{\QM}\lesssim \sum_{i\not=j}\frac{\lambda^{\frac{\mu}{2}}}{(1+\lambda|x-\xi_i|)^\mu}. $$ Moreover, for any $x\in \Omega_j$ and $i\not=j $ we get for $0<\delta<\frac{\mu}{2}$ and by using Lemma \ref{LEM1}
    \begin{equation*}
        \begin{aligned}
            \sum_{i\not=j}\frac{1}{(1+\lambda|x-\xi_i|)^\mu}&\lesssim \sum_{i\not=j}\frac{1}{(1+\lambda|x-\xi_i|)^{\frac{\mu}{2}}}\frac{1}{(1+\lambda|x-\xi_j|)^{\frac{\mu}{2}}}\\
            &\lesssim \sum_{i\not=j}\frac{1}{\lambda^\delta|\xi_i-\xi_j|^\delta}\frac{1}{(1+\lambda|x-\xi_j|)^{\mu-\delta}}\\
            &\lesssim\frac{1}{\lambda^\delta}\frac{1}{(1+\lambda|x-\xi_j|)^{\mu-\delta}}.
        \end{aligned}
    \end{equation*}
    Hence
    \begin{equation*}
        \begin{aligned}
            |(A)|&\lesssim \frac{1}{\lambda^\alpha} \sum_{j=1}^k\frac{\lambda^{\frac{\mu}{2}}}{(1+\lambda|x-\xi_j|)^{4-\alpha-\eta}}\sum_{h=1}^k(\mathcal{U}_{\lambda,\xi_h})^{\TM}+\frac{1}{\lambda^\delta}\sum_{j=1}^k\frac{\lambda^{\frac{\mu}{2}}}{(1+\lambda|x-\xi_j|)^{\mu-\delta}}\sum_{h=1}^k(\mathcal{U}_{\lambda,\xi_h})^{\TM}\\
            &\lesssim \frac{1}{\lambda^\alpha}\lambda^3 \sum_{j=1}^k\frac{1}{(1+\lambda|x-\xi_j|)^{4-\alpha-\eta}}\sum_{h=1}^k\frac{1}{(1+\lambda|x-\xi_h|)^{6-\mu}}+\\
            &+\frac{1}{\lambda^\delta}\lambda^3\sum_{j=1}^k\frac{1}{(1+\lambda|x-\xi_j|)^{\mu-\delta}}\sum_{h=1}^k\frac{1}{(1+\lambda|x-\xi_h|)^{6-\mu}}\\
            &\lesssim \bigg(\frac{1}{\lambda}\bigg)\lambda^3\sum_{j=1}^k \frac{1}{(1+\lambda|x-\xi_j|)^{3+\tau}}.
        \end{aligned}
    \end{equation*}
    Indeed, by choosing $\alpha=\delta:=1$ we have that
    $$4-\alpha-\eta+6-\mu:=4-1-\eta+6-\mu>3+\tau$$ and $$\mu-\delta+6-\mu:=\mu-1+6-\mu>3+\tau.$$
    Hence, $$\|(A)\|_{**}\lesssim \bigg(\frac{1}{\lambda}\bigg). $$
    For the term $(B)$ we reason as for the second term in $(A)$ and hence we get also that $$\|(B)\|_{**}\lesssim \bigg(\frac{1}{\lambda}\bigg).$$
\end{proof}
The previous results allow us to apply a fixed point argument in order to solve the non-linear system
\begin{equation}\label{perpen}
    \mathcal{L}(\phi_1,\phi_2)=\mathcal{E}+\mathcal{N}(\phi_1,\phi_2) \quad \hbox{in}\, \, X\times K^\perp,
\end{equation}
which corresponds to \eqref{pio}. More precisely, we have:
\begin{proposition}\label{unicità}
    There exists $\beta_0<0$ such that for any $\beta\in [\beta_0,0)$ and $\lambda$ such that $\frac{1}{\lambda}\in (0,e^{-\frac{1}{\sqrt{|\beta|}}}),$ system \eqref{perpen} has an unique solution $(\phi_1,\phi_2)\in X\times K^\perp.$ Furthermore, there exists a constant $C>0$ such that 
    $$\|\phi_1\|_*+\|\phi_2\|_*\leq C\frac{|\beta|}{\lambda}.$$
\end{proposition}
The proof of this result follows by Proposition \ref{inveritibilità} and Proposition \ref{errore} and it is standard in the literature concerning Lyapunov-Schmidt methods, so we omit it for the sake of simplicity. Let us just mention that the linear terms contained in $\mathcal{N}$ do not cause any trouble since they all have the parameter $\beta$ in front.
\subsection{The reduced problem}
Consider $(\phi_1,\phi_2)$ provided by Proposition \ref{unicità}. Thus, $$\mathcal{E}_2+\mathcal{N}_2(\phi_1,\phi_2)-\mathcal{L}_2(\phi_1,\phi_2)=\mathfrak{c}Z,$$ for some constant $\mathfrak{c}$ depending on $\lambda.$ Hence, seeing that they also satisfy \eqref{pi} is equivalent to find $\lambda=\lambda(\beta)$ such that $\mathfrak{c}(\lambda)=0.$ Testing in the equation above, we see that this constant is given by 
\begin{equation}
    \mathfrak{c}(\lambda)=\frac{\int_\RQ(\mathcal{E}_2+\mathcal{N}_2(\phi_1,\phi_2)-\mathcal{L}_2(\phi_1,\phi_2)) Z\, dx}{\int_\RQ| Z|^2\, dx}.
\end{equation}
\begin{proposition}\label{ridotto}
    There exists $\mathfrak{a},\mathfrak{b}>0$ such that
    $$\mathfrak{c}(\lambda)=-\mathfrak{a}\frac{1}{\lambda^2}(1+o(1))+\mathfrak{b}\beta\frac{1}{\lambda^2}\ln\bigg(\frac{1}{\lambda}\bigg)(1+o(1))$$ where $\beta<0$ and $|\beta|$ small enough.
\end{proposition}
\begin{proof}
    Let's start by computing the denominator.\\
    \begin{equation*}
        \begin{aligned}
            \int_\RQ|Z(x)|^2\, dx&=\lambda^2\int_\RQ\frac{(1-\lambda^2|x-\xi_j|^2)^2}{(1+\lambda^2|x-\xi_j|^2)^4}\, dx =\lambda^2 \omega_4\int_0^R\frac{(1-\lambda^2r^2)r^3}{(1+\lambda^2r^2)^4}\, dr\\
            &=\frac{\lambda^2\omega_4}{2\lambda^4}\int_0^{\lambda^2R^2}\frac{(1-s)^2s}{(1+s)^4}\, ds \sim\frac{1}{\lambda^2}\ln(\lambda).
        \end{aligned}
    \end{equation*}
    Now, we compute the numerator.\\
    \textbf{STEP 1:}
    \begin{equation*}
        \begin{aligned}
            \int_\RQ\mathcal{E}_2\cdot Z\, dx &=\underbrace{ \int_\RQ[g(V)-\sum_{j=1}^kg(\Uj)]\cdot Z\, dx}_{I_1}+\underbrace{\int_\RQ\beta \mathcal{U}^2VZ\, dx}_{I_2}.
        \end{aligned}
    \end{equation*}
    Reasoning as in \cite{CMP}, we can obtain that
    $$I_2=c_2\beta\frac{1}{\lambda^2}\ln\bigg(\frac{1}{\lambda^2}\bigg)+O\bigg(|\beta|\frac{1}{\lambda^2}\bigg) \quad \hbox{for some}\, \, c_2>0.$$
    On the other hand, by symmetry
    \begin{equation*}
        \begin{aligned}
            I_1&=k\int_\RQ[g(V)-\sum_{j=1}^kg(\Uj)]\cdot \Zo\, dx\\
            &=k\int_\RQ (|x|^{-\mu}*|V|^{\QM})|V|^{\TM}\Zo-k\sum_{j=1}^k\int_\RQ(|x|^{-\mu}*|\Uj|^{\QM})|\Uj|^{\TM}\Zo\\
            &=k^2 \int_{\Omega_1} (|x|^{-\mu}*|V|^{\QM})|V|^{\TM}\Zo-k^2\sum_{j=1}^k\int_{\Omega_1}(|x|^{-\mu}*|\Uj|^{\QM})|\Uj|^{\TM}\Zo.
        \end{aligned}
    \end{equation*}
Since for any $x\in\Omega_1$
\begin{equation}
\begin{aligned}
    V^{\QM}:=&\bigg(\U+\sum_{j=2}^k\Uj\bigg)^{\QM}=\\
    &=(\U)^{\QM}+\bigg(\QM\bigg)(\U)^{\TM}\sum_{j=2}^k\Uj+O\bigg((\U)^{\DQ}\sum_{j=2}^k(\Uj)^{\DQ}\bigg)
\end{aligned}
\end{equation}
and 
\begin{equation}
\begin{aligned}
     V^{\TM}&:=\bigg(\U+\sum_{j=2}^k\Uj\bigg)^{\TM}=\\
     &=(\U)^{\TM}+\bigg(\TM\bigg)(\U)^{\DM}\sum_{j=2}^k\Uj+O\bigg((\U)^{\UQ}\sum_{j=2}^k(\Uj)^{\DQ}\bigg)
\end{aligned}
\end{equation}
it results that
\begin{equation*}
    \begin{aligned}
       &I_1=\underbrace{k^2\bigg(\TM\bigg)\int_{\Omega_1}(|x|^{-\mu}*\U^{\QM})\U^{\DM}\sum_{j=2}^k\Uj\Zo}_{(A)}+\\
        &+O\bigg(\underbrace{\int_{\Omega_1}(|x|^{-\mu}*\U^{\QM})\U^{\UQ}\sum_{j=2}^k\Uj^{\DQ}\Zo}_{(B)}\bigg)+\\
        &+\underbrace{k^2 \bigg(\QM\bigg)\int_{\Omega_1}(|x|^{-\mu}*\U^{\TM}\sum_{j=2}^k\Uj)\U^{\TM}\Zo}_{(C)}+\\
        &+k^2 \bigg(\QM\bigg)\bigg(\TM\bigg)\underbrace{\int_{\Omega_1}(|x|^{-\mu}*\U^{\TM}\sum_{j=2}^k\Uj)\U^{\DM}\sum_{j=2}^k\Uj\Zo}_{(D)}+\\
        &+ O\bigg(\underbrace{\int_{\Omega_1}(|x|^{-\mu}*\U^{\TM}\sum_{j=2}^k\Uj)\U^{\UQ}\sum_{j=2}^k\Uj^{\DQ}\Zo}_{(E)}\bigg)+\\
        &+ O\bigg(\underbrace{\int_{\Omega_1}(|x|^{-\mu}*\U^{\DQ}\sum_{j=2}^k\Uj^{\DQ})\U^{\TM}\Zo}_{(F)}\bigg)+\\
        &+O\bigg(\underbrace{\int_{\Omega_1}(|x|^{-\mu}*\U^{\DQ}\sum_{j=2}^k\Uj^{\DQ})\U^{\DM}\sum_{j=2}^k\Uj\Zo}_{(G)}\bigg)+\\
        &+O\bigg(\underbrace{\int_{\Omega_1}(|x|^{-\mu}*\U^{\DQ}\sum_{j=2}^k\Uj^{\DQ})\U^{\UQ}\sum_{j=2}^k\Uj^{\DQ}\Zo}_{(H)}\bigg)\\
        &\underbrace{-k^2\sum_{j=2}^k\int_{\Omega_1}(|x|^{-\mu}*|\Uj|^{\QM})|\Uj|^{\TM}\Zo}_{(I)}.
    \end{aligned}
\end{equation*}
By Remark \ref{REM}, it results that 
$$|x|^{-\mu}*\U^{\QM}=\sigma_{4,\mu}\frac{\lambda^{\frac{\mu}{2}}}{(1+\lambda^2|x-\xi_1|^2)^{\frac{\mu}{2}}}, \quad \hbox{where }\, \, \sigma_{4,\mu}>0.$$
So, 
\begin{equation*}
    \begin{aligned}
        (A)&=k^2\bigg(\TM\bigg)\sigma_{4,\mu}\alpha_{4,\mu}^{\QM}\int_{\Omega_1}\frac{\lambda^2}{(1+\lambda^2|x-\xi_1|^2)^2}\frac{\lambda}{(1+\lambda^2|x-\xi_j|^2)}\lambda\frac{1-\lambda^2|x-\xi_1|^2}{(1+\lambda^2|x-\xi_1|^2)^2}\, dx \\
        &=k^2\bigg(\TM\bigg)\sigma_{4,\mu}\alpha_{4,\mu}^{\QM}\int_{\Omega_1}\lambda^4 \frac{1-\lambda^2|x-\xi_1|^2}{(1+\lambda^2|x-\xi_1|^2)^4}\frac{1}{(1+\lambda^2|x-\xi_j|^2)}\, dx \\
        &=k^2\bigg(\TM\bigg)\sigma_{4,\mu}\alpha_{4,\mu}^{\QM}\int_{\lambda(\Omega_1-\xi_1)}\frac{1-|x|^2}{(1+|x|^2)^4}\frac{1}{(1+\lambda^2|\frac{x}{\lambda}+\xi_1-\xi_j|^2)}\, dx \\
        &=k^2\bigg(\TM\bigg)\sigma_{4,\mu}\alpha_{4,\mu}^{\QM} \int_\RQ \frac{1-|x|^2}{(1+|x|^2)^4}\frac{1}{\lambda^2(\frac{1}{\lambda^2}+|\frac{x}{\lambda}+\xi_1-\xi_j|^2)}\, dx \\
        &-k^2\bigg(\TM\bigg)\sigma_{4,\mu}\alpha_{4,\mu}^{\QM} \int_{\RQ\setminus\lambda(\Omega_1-\xi_1)} \frac{1-|x|^2}{(1+|x|^2)^4}\frac{1}{(1+|\lambda(\xi_1-\xi_j)-x|^2)}\, dx \\
        &=k^2\bigg(\TM\bigg)\sigma_{4,\mu}\alpha_{4,\mu}^{\QM} \sum_{j=2}^k\frac{1}{\lambda^2|\xi_1-\xi_j|^2}\int_\RQ \frac{1-|x|^2}{(1+|x|^2)^4}\, dx +O\bigg(\sum_{j=2}^k\frac{1}{\lambda^{2+\sigma}|\xi_1-\xi_j|^{2+\sigma}}\bigg)
    \end{aligned}
\end{equation*}
where $\sigma>0$ is small enough and 
$$\int_\RQ \frac{1-|x|^2}{(1+|x|^2)^4}\, dx=-\frac{1}{3}\int_\RQ\frac{1}{(1+|x|^2)^3}\, dx$$ (see proof of Lemma 3.3 in \cite{CV}).
\begin{equation*}
    \begin{aligned}
        (B)&\lesssim \int_\RQ\frac{\lambda^{\frac{\mu}{2}}}{(1+\lambda^2|x-\xi_1|^2)^{\frac{\mu}{2}}}\frac{\lambda^{1-\frac{\mu}{4}}}{(1+\lambda^2|x-\xi_1|^2)^{1-\frac{\mu}{4}}}\frac{\lambda^{2-\frac{\mu}{4}}}{(1+\lambda^2|x-\xi_j|^2)^{2-\frac{\mu}{4}}}\lambda\frac{1-\lambda^2|x-\xi_1|^2}{(1+\lambda^2|x-\xi_1|^2)^2}\, dx\\
        &\lesssim\int_\RQ \lambda^4 \frac{1-\lambda^2|x-\xi_1|^2}{(1+\lambda^2|x-\xi_1|^2)^{3+\frac{\mu}{4}}}\frac{1}{(1+\lambda^2|x-\xi_j|^2)^{2-\frac{\mu}{4}}}\\
        &\lesssim\int_\RQ\frac{1-|x|^2}{(1+|x|^2)^{3+\frac{\mu}{4}}}\frac{1}{\lambda^{\QM}(\frac{1}{\lambda^2}+|\frac{x}{\lambda}+\xi_1-\xi_j|^2)^{\DQ}}=O\bigg(\sum_{j=2}^k\frac{1}{\lambda^{2+\sigma}|\xi_1-\xi_j|^{2+\sigma}}\bigg).
    \end{aligned}
\end{equation*}
\begin{equation*}
    \begin{aligned}
        (C)&=k^2 \int_{\Omega_1} \bigg(\frac{\partial}{\partial\lambda}(|x|^{-\mu}*\U^{\QM})\bigg)\U^{\TM}\sum_{j=2}^k\Uj\cdot\lambda\, dx \\
        &=k^2\frac{\mu}{2}\sigma_{4,\mu}\alpha^{\QM}\sum_{j=2}^k\int_{\Omega_1}\lambda^4\frac{1-\lambda^2|x-\xi_1|^2}{(1+\lambda^|x-\xi_1|^2)^4}\frac{1}{(1+\lambda^2|x-\xi_j|^2)}\, dx\\
        &=k^2\frac{\mu}{2}\sigma_{4,\mu}\alpha^{\QM}\sum_{j=2}^k\int_\RQ \frac{1-|x|^2}{(1+|x|^2)^4}\frac{1}{\lambda^2(\frac{1}{\lambda^2}+|\frac{x}{\lambda}+\xi_1-\xi_j|^2)}\, dx \\
        &-k^2\frac{\mu}{2}\sigma_{4,\mu}\alpha^{\QM}\sum_{j=2}^k\int_{\RQ\setminus\lambda(\Omega_1-\xi_1)}\frac{1-|x|^2}{(1+|x|^2)^4}\frac{1}{(1+|\lambda(\xi_1-\xi_j)-x|^2)}\, dx \\
        &=k^2\frac{\mu}{2}\sigma_{4,\mu}\alpha^{\QM}\sum_{j=2}^k\frac{1}{\lambda^2|\xi_1-\xi_j|^2}\int_\RQ \frac{1-|x|^2}{(1+|x|^2)^4}\, dx +O\bigg(\sum_{j=2}^k\frac{1}{\lambda^{2+\sigma}|\xi_1-\xi_j|^{2+\sigma}}\bigg)
    \end{aligned}
\end{equation*}
where, as before, 
$$\int_\RQ \frac{1-|x|^2}{(1+|x|^2)^4}\, dx=-\frac{1}{3}\int_\RQ\frac{1}{(1+|x|^2)^3}\, dx$$ 
Now, since by Lemma \ref{LEM1} and Lemma \ref{LEM2}
$$(|x|^{-\mu}*  \U^{\TM}\sum_{j=2}^k\Uj)\lesssim\sum_{j=2}^k\frac{1}{\lambda^2|\xi_1-\xi_j|^2}\frac{\lambda^\frac{\mu}{2}}{(1+\lambda|x-\xi_1|)^{2-\eta}}$$ with $\eta>0$ small enough, we have that 
\begin{equation*}
    \begin{aligned}
        (D)&\lesssim \sum_{j=2}^k\frac{1}{\lambda^2|\xi_1-\xi_j|^2}\int_\RQ \frac{\lambda^2}{(1+\lambda|x-\xi_1|)^{4-\frac{\mu}{2}-\eta}}\frac{\lambda}{(1+\lambda|x-\xi_j|)^2}\lambda\frac{1-\lambda^2|x-\xi_1|^2}{(1+\lambda^2|x-\xi_1|^2)^2}\, dx \\
        &\lesssim \sum_{j=2}^k\frac{1}{\lambda^2|\xi_1-\xi_j|^2}\int_\RQ \lambda^4\frac{1-\lambda^2|x-\xi_1|^2}{(1+\lambda|x-\xi_1|)^{10-\mu-\eta}}\frac{1}{(1+\lambda|x-\xi_j|)^2}\, dx \\
        &\lesssim \sum_{j=2}^k\frac{1}{\lambda^2|\xi_1-\xi_j|^2}\int_\RQ\frac{1-|x|^2}{(1+|x|)^{10-\mu-\eta}}\frac{1}{\lambda^2(\frac{1}{\lambda}+|\frac{x}{\lambda}+\xi_1-\xi_j|)^2}\, dx \\
        &=O\bigg(\sum_{j=2}^k\frac{1}{\lambda^{2+\sigma}|\xi_1-\xi_j|^{2+\sigma}}\bigg).
    \end{aligned}
\end{equation*}
and 
\begin{equation*}
    \begin{aligned}
        (E)&\lesssim \sum_{j=2}^k\frac{1}{\lambda^2|\xi_1-\xi_j|^2}\int_\RQ \frac{\lambda^{1+\frac{\mu}{4}}}{(1+\lambda|x-\xi_1|)^{4-\frac{\mu}{2}-\eta}}\frac{\lambda^{\DQ}}{(1+\lambda|x-\xi_j|)^{\QM}}\lambda\frac{1-\lambda^2|x-\xi_1|^2}{(1+\lambda^2|x-\xi_1|^2)^2}\, dx \\
        &\lesssim \sum_{j=2}^k\frac{1}{\lambda^2|\xi_1-\xi_j|^2}\int_\RQ \lambda^4\frac{1-\lambda^2|x-\xi_1|^2}{(1+\lambda|x-\xi_1|)^{8-\frac{\mu}{2}-\eta}}\frac{1}{(1+\lambda|x-\xi_j|)^{\QM}}\, dx \\
        &\lesssim \sum_{j=2}^k\frac{1}{\lambda^2|\xi_1-\xi_j|^2}\int_\RQ \frac{1-|x|^2}{(1+|x|)^{8-\frac{\mu}{2}-\eta}}\frac{1}{\lambda^{\QM}(\frac{1}{\lambda}+|\frac{x}{\lambda}+\xi_1-\xi_j|)^{\QM}}\, dx \\
        &=O\bigg(\sum_{j=2}^k\frac{1}{\lambda^{2+\sigma}|\xi_1-\xi_j|^{2+\sigma}}\bigg).
    \end{aligned}
\end{equation*}
Now, since by Lemma \ref{LEM1} and Lemma \ref{LEM2}
$$(|x|^{-\mu}*\U^{\DQ}\Uj^{\DQ})\lesssim\frac{\lambda^{\frac{\mu}{2}}}{\lambda^{2+\sigma}|\xi_1-\xi_j|^{2+\sigma}}\bigg(\frac{1}{(1+\lambda|x-\xi_1|)^{2-\sigma-\eta}}+\frac{1}{(1+\lambda|x-\xi_j|)^{2-\sigma-\eta}}\bigg)$$ we have that
\begin{equation*}
    \begin{aligned}
        (F)&\lesssim \sum_{j=2}^k \frac{1}{\lambda^{2+\sigma}|\xi_1-\xi_j|^{2+\sigma}}\int_\RQ \frac{\lambda^{\frac{\mu}{2}}}{(1+\lambda|x-\xi_1|)^{2-\sigma-\eta}}\frac{\lambda^{\TM}}{(1+\lambda|x-\xi_1|)^{6-\mu}}\lambda\frac{1-\lambda^2|x-\xi_1|^2}{(1+\lambda^2|x-\xi_1|^2)^2}\, dx\\
        &+\sum_{j=2}^k \frac{1}{\lambda^{2+\sigma}|\xi_1-\xi_j|^{2+\sigma}}\int_\RQ \frac{\lambda^{\frac{\mu}{2}}}{(1+\lambda|x-\xi_j|)^{2-\sigma-\eta}}\frac{\lambda^{\TM}}{(1+\lambda|x-\xi_1|)^{6-\mu}}\lambda\frac{1-\lambda^2|x-\xi_1|^2}{(1+\lambda^2|x-\xi_1|^2)^2}\, dx\\
        &\lesssim \sum_{j=2}^k \frac{1}{\lambda^{2+\sigma}|\xi_1-\xi_j|^{2+\sigma}}\int_\RQ\lambda^4 \frac{1-\lambda^2|x-\xi_1|^2}{(1+\lambda|x-\xi_1|)^{12-\mu-\sigma-\eta}}\, dx+\\
        &+ \sum_{j=2}^k \frac{1}{\lambda^{2+\sigma}|\xi_1-\xi_j|^{2+\sigma}}\int_\RQ\lambda^4 \frac{1-\lambda^2|x-\xi_1|^2}{(1+\lambda|x-\xi_1|)^{10-\mu}}\frac{1}{(1+\lambda|x-\xi_j|)^{2-\sigma-\eta}}\, dx\\
        &=O\bigg(\sum_{j=2}^k\frac{1}{\lambda^{2+\sigma}|\xi_1-\xi_j|^{2+\sigma}}\bigg).
    \end{aligned}
\end{equation*}
\begin{equation*}
    \begin{aligned}
        (G)&\lesssim \sum_{j=2}^k \frac{1}{\lambda^{2+\sigma}|\xi_1-\xi_j|^{2+\sigma}}\int_\RQ \frac{\lambda^2}{(1+\lambda|x-\xi_1|)^{6-\mu-\sigma-\eta}}\frac{\lambda}{(1+\lambda|x-\xi_j|)^2}\lambda\frac{1-\lambda^2|x-\xi_1|^2}{(1+\lambda^2|x-\xi_1|^2)^2}\\
        &+ \sum_{j=2}^k \frac{1}{\lambda^{2+\sigma}|\xi_1-\xi_j|^{2+\sigma}}\int_\RQ \frac{\lambda^{\frac{\mu}{2}}}{(1+\lambda|x-\xi_j|)^{2-\sigma-\eta}}\frac{\lambda^{\DM}}{(1+\lambda|x-\xi_1|)^{4-\mu}}\frac{\lambda}{(1+\lambda|x-\xi_j|)^2}\cdot \\
        &\cdot \lambda\frac{1-\lambda^2|x-\xi_1|^2}{(1+\lambda^2|x-\xi_1|^2)^2}\\
        &\lesssim\sum_{j=2}^k \frac{1}{\lambda^{2+\sigma}|\xi_1-\xi_j|^{2+\sigma}}\int_\RQ\lambda^4 \frac{1-\lambda^2|x-\xi_1|^2}{(1+\lambda|x-\xi_1|)^{10-\mu-\sigma-\eta}}\frac{1}{(1+\lambda|x-\xi_j|)^2}\, dx+\\
        &+\sum_{j=2}^k \frac{1}{\lambda^{2+\sigma}|\xi_1-\xi_j|^{2+\sigma}}\int_\RQ\lambda^4 \frac{1-\lambda^2|x-\xi_1|^2}{(1+\lambda|x-\xi_1|)^{8-\mu}}\frac{1}{(1+\lambda|x-\xi_j|)^{4-\sigma-\eta}}\, dx\\
        &=O\bigg(\sum_{j=2}^k\frac{1}{\lambda^{2+\sigma}|\xi_1-\xi_j|^{2+\sigma}}\bigg).
    \end{aligned}
\end{equation*}
and 
\begin{equation*}
    \begin{aligned}
        (H)&\lesssim \sum_{j=2}^k \frac{1}{\lambda^{2+\sigma}|\xi_1-\xi_j|^{2+\sigma}}\int_\RQ \frac{\lambda^{\frac{\mu}{2}}}{(1+\lambda|x-\xi_1|)^{2-\sigma-\eta}}\frac{\lambda^{\UQ}}{(1+\lambda|x-\xi_1|)^{\DM}}\frac{\lambda^{\DQ}}{(1+\lambda|x-\xi_j|)^{\QM}}\cdot \\
        &\cdot \lambda\frac{1-\lambda^2|x-\xi_1|^2}{(1+\lambda^2|x-\xi_1|^2)^2}\\
        &+\sum_{j=2}^k \frac{1}{\lambda^{2+\sigma}|\xi_1-\xi_j|^{2+\sigma}}\int_\RQ \frac{\lambda^{\frac{\mu}{2}}}{(1+\lambda|x-\xi_j|)^{2-\sigma-\eta}}\frac{\lambda^{\UQ}}{(1+\lambda|x-\xi_1|)^{\DM}}\frac{\lambda^{\DQ}}{(1+\lambda|x-\xi_j|)^{\QM}}\cdot \\
        &\cdot\lambda\frac{1-\lambda^2|x-\xi_1|^2}{(1+\lambda^2|x-\xi_1|^2)^2}\\
        &\lesssim \sum_{j=2}^k \frac{1}{\lambda^{2+\sigma}|\xi_1-\xi_j|^{2+\sigma}}\int_\RQ\lambda^4 \frac{1-\lambda^2|x-\xi_1|^2}{(1+\lambda|x-\xi_1|)^{8-\frac{\mu}{2}-\sigma-\eta}}\frac{1}{(1+\lambda|x-\xi_j|)^{\QM}}\, dx+\\
        &+\sum_{j=2}^k \frac{1}{\lambda^{2+\sigma}|\xi_1-\xi_j|^{2+\sigma}}\int_\RQ\lambda^4 \frac{1-\lambda^2|x-\xi_1|^2}{(1+\lambda|x-\xi_1|)^{6-\frac{\mu}{2}}}\frac{1}{(1+\lambda|x-\xi_j|)^{6-\frac{\mu}{2}-\sigma-\eta}}\, dx\\
        &=O\bigg(\sum_{j=2}^k\frac{1}{\lambda^{2+\sigma}|\xi_1-\xi_j|^{2+\sigma}}\bigg).
    \end{aligned}
\end{equation*}
Finally, 
\begin{equation*}
    \begin{aligned}
    |(I)|&\lesssim \sum_{j=2}^k\int_\RQ \frac{\lambda^{\frac{\mu}{2}}}{(1+\lambda^2|x-\xi_j|^2)^{\frac{\mu}{2}}}\frac{\lambda^{\TM}}{(1+\lambda^2|x-\xi_j|^2)^{\TM}}\lambda\frac{1-\lambda^2|x-\xi_1|^2}{(1+\lambda^2|x-\xi_1|^2)^2}\, dx \\
    &\lesssim\sum_{j=2}^k\int_\RQ\lambda^4\frac{1-\lambda^2|x-\xi_1|^2}{(1+\lambda^2|x-\xi_1|^2)^2}\frac{1}{(1+\lambda^2|x-\xi_j|^2)^{3}}\, dx \\
    &=O\bigg(\sum_{j=2}^k\frac{1}{\lambda^{2+\sigma}|\xi_1-\xi_j|^{2+\sigma}}\bigg).
    \end{aligned}
\end{equation*}
Therefore, 
\begin{equation}
    \begin{aligned}
        \int_\RQ\mathcal{E}_2Z\, dx &=3k^2\sigma_{4,\mu}\alpha_{4,\mu}c_1\sum_{j=2}^k\frac{1}{\lambda^2|\xi_1-\xi_j|^2}+O\bigg(\frac{1}{\lambda^{2+\sigma}|\xi_1-\xi_j|^{2+\sigma}}\bigg)\\
        &+C_2\beta\frac{1}{\lambda^2}\ln\bigg(\frac{1}{\lambda}\bigg)+O\bigg(|\beta|\frac{1}{\lambda^2}\bigg)
    \end{aligned}
\end{equation}
where $$c_1:=-\frac{1}{3}\int_\RQ \frac{1}{(1+|x|^2)^3}\, dx.$$
\textbf{STEP 2:}
Now, 
\begin{equation*}
    \begin{aligned}
        \int_\RQ\mathcal{L}_2(\phi_1,\phi_2)Z\, dx &=k\int_\RQ\mathcal{L}_2(\phi_1,\phi_2)\Zo=k\int_\RQ[-\Delta\phi_2-g'(V)\phi_2]\Zo\\
        &=k\int_\RQ[g'(\U)-g'(V)]\phi_2\Zo\\
        &=k\bigg(\QM\bigg)\underbrace{\int_\RQ\bigg[|x|^{-\mu}*\bigg(\U^{\TM}-(\sum_{j=1}^k\Uj)^{\TM}\bigg)\phi_2\bigg]\U^{\TM}\Zo}_{(L)}\\
        &+k\bigg(\QM\bigg)\underbrace{\int_\RQ\bigg[|x|^{-\mu}*(\sum_{j=1}^k\Uj)^{\TM}\phi_2\bigg](\U^{\TM}-(\sum_{j=1}^k\Uj)^{\TM})\Zo}_{(M)}\\
        &+k\bigg(\TM\bigg)\underbrace{\int_\RQ\bigg[|x|^{-\mu}*\bigg(\U^{\QM}-(\sum_{j=1}^k\Uj)^{\QM}\bigg)\bigg]\U^{\DM}\phi_2\Zo}_{(N)}\\
        &+k\bigg(\TM\bigg)\underbrace{\int_\RQ\bigg[|x|^{-\mu}*(\sum_{j=1}^k\Uj)^{\QM}\bigg](\U^{\DM}-(\sum_{j=1}^k\Uj)^{\DM})\phi_2\Zo.}_{(O)}
    \end{aligned}
\end{equation*}
Since 
$$|(\sum_{j=1}^k\Uj)^{\TM}-\U^{\TM}|\lesssim\U^{\DM}\sum_{j=2}^k\Uj+\sum_{j=2}^k\Uj^{\TM}$$ it results that
$$\U^{\DM}\Uj\lesssim \frac{\lambda^{\DM}}{(1+\lambda|x-\xi_1|)^{4-\mu}}\frac{\lambda}{(1+\lambda|x-\xi_j|)^2}\lesssim\frac{1}{\lambda^\alpha|\xi_1-\xi_j|^\alpha}\frac{\lambda^{\TM}}{(1+\lambda|x-\xi_1|)^{6-\mu-\alpha}}$$ with $0<\alpha<\min\{2,4-\mu\}$, and so $$\U^{\TM}\sum_{j=2}^k\Uj\lesssim\frac{1}{\lambda^\alpha}\frac{\lambda^{\TM}}{(1+\lambda|x-\xi_1|)^{6-\mu-\alpha}}.$$
Therefore, 
\begin{equation*}
    \begin{aligned}
        |x|^{-\mu}*(\U^{\DM}\sum_{j=2}^k\Uj)\phi_2&\lesssim \int_\RQ\frac{1}{|x-y|^\mu}\frac{1}{\lambda^\alpha}\frac{\lambda^{\TM}}{(1+\lambda|y-\xi_1|)^{6-\mu-\alpha}}\phi_2\, dy\\
        &\lesssim\frac{\|\phi_2\|_*}{\lambda^\alpha}\int_\RQ\frac{1}{|x-y|^\mu}\frac{\lambda^{\TM}}{(1+\lambda|y-\xi_1|)^{6-\mu-\alpha}}\sum_{j=1}^k\frac{\lambda}{(1+\lambda|y-\xi_j|)^{1+\tau}}\, dy\\
        &\lesssim\frac{\|\phi_2\|_*}{\lambda^{\alpha}}\int_\RQ \frac{1}{|x-y|^\mu}\frac{\lambda^{\QM}}{(1+\lambda|y-\xi_1|)^{7+\tau-\mu-\alpha}}\, dy\\
        &+\frac{\|\phi_2\|_*}{\lambda^{\alpha+1+\tau}}\int_\RQ \frac{1}{|x-y|^\mu}\frac{\lambda^{\QM}}{(1+\lambda|y-\xi_1|)^{6-\mu-\alpha}}\, dy\\
        &\lesssim \frac{\|\phi_2\|_*}{\lambda^{\alpha}}\frac{\lambda^{\frac{\mu}{2}}}{(1+\lambda|x-\xi_1|)^{3+\tau-\alpha-\eta}}+\frac{\|\phi_2\|_*}{\lambda^{\alpha+1+\tau}}\frac{\lambda^{\frac{\mu}{2}}}{(1+\lambda|x-\xi_1|)^{2-\alpha-\eta}}
    \end{aligned}
\end{equation*}
with $\eta>0$ small enough.\\
Moreover, 
\begin{equation*}
    \begin{aligned}
        |x|^{-\mu}*\sum_{j=2}^k\Uj^{\TM}\phi_2&\lesssim \|\phi_2\|_*\int_\RQ\frac{1}{|x-y|^\mu}\sum_{j=2}^k\frac{\lambda^{\TM}}{(1+\lambda|y-\xi_j|)^{6-\mu}}\sum_{i=1}^k\frac{\lambda}{(1+\lambda|y-\xi_i|)^{1+\tau}}\, dy\\
        &\lesssim \frac{\|\phi_2\|_*}{\lambda^{1+\tau}}\int_\RQ \frac{1}{|x-y|^\mu}\frac{\lambda^{\QM}}{(1+\lambda|y-\xi_1|)^{6-\mu}}\, dy+\\
        &+\frac{\|\phi_2\|_*}{\lambda^{1+\tau}}\sum_{j=2}^k\int_\RQ \frac{1}{|x-y|^\mu}\frac{\lambda^{\QM}}{(1+\lambda|y-\xi_j|)^{6-\mu}}\, dy\\
        &\lesssim \frac{\|\phi_2\|_*}{\lambda^{1+\tau}}\frac{\lambda^{\frac{\mu}{2}}}{(1+\lambda|x-\xi_1|)^{2-\eta}}+\frac{\|\phi_2\|_*}{\lambda^{1+\tau}}\sum_{j=2}^k\frac{\lambda^{\frac{\mu}{2}}}{(1+\lambda|x-\xi_j|)^{2-\eta}}.
    \end{aligned}
\end{equation*}
Therefore, 
\begin{equation*}
    \begin{aligned}
        (L)&\lesssim \frac{\|\phi_2\|_*}{\lambda^{\alpha}}\int_\RQ \frac{\lambda^{\frac{\mu}{2}}}{(1+\lambda|x-\xi_1|)^{3+\tau-\alpha-\eta}}\frac{\lambda^{\TM}}{(1+\lambda|x-\xi_1|)^{6-\mu}}\lambda\frac{1-\lambda^2|x-\xi_1|^2}{(1+\lambda^2|x-\xi_1|^2)^2}\, dx\\
        &+\frac{\|\phi_2\|_*}{\lambda^{\alpha+1+\tau}}\int_\RQ \frac{\lambda^{\frac{\mu}{2}}}{(1+\lambda|x-\xi_1|)^{2-\eta-\alpha}}\frac{\lambda^{\TM}}{(1+\lambda|x-\xi_1|)^{6-\mu}}\lambda\frac{1-\lambda^2|x-\xi_1|^2}{(1+\lambda^2|x-\xi_1|^2)^2}\, dx\\
        &+ \frac{\|\phi_2\|_*}{\lambda^{1+\tau}}\int_\RQ\frac{\lambda^{\frac{\mu}{2}}}{(1+\lambda|x-\xi_1|)^{2-\eta}}\frac{\lambda^{\TM}}{(1+\lambda|x-\xi_1|)^{6-\mu}}\lambda\frac{1-\lambda^2|x-\xi_1|^2}{(1+\lambda^2|x-\xi_1|^2)^2}\, dx\\
        &+ \frac{\|\phi_2\|_*}{\lambda^{1+\tau}}\int_\RQ\frac{\lambda^{\frac{\mu}{2}}}{(1+\lambda|x-\xi_j|)^{2-\eta}}\frac{\lambda^{\TM}}{(1+\lambda|x-\xi_1|)^{6-\mu}}\lambda\frac{1-\lambda^2|x-\xi_1|^2}{(1+\lambda^2|x-\xi_1|^2)^2}\, dx\\
        &\lesssim\frac{\|\phi_2\|_*}{\lambda^{\alpha}}\int_\RQ \frac{\lambda^4}{(1+\lambda|x-\xi_1|)^{11+\tau-\mu-\alpha-\eta}}+\frac{\|\phi_2\|_*}{\lambda^{\alpha+1+\tau}}\int_\RQ \frac{\lambda^4}{(1+\lambda|x-\xi_1|)^{10-\mu-\alpha-\eta}}\\
        &+\frac{\|\phi_2\|_*}{\lambda^{1+\tau}}\int_\RQ \frac{\lambda^4}{(1+\lambda|x-\xi_1|)^{10-\mu-\eta}}+\frac{\|\phi_2\|_*}{\lambda^{1+\tau+\delta}}\int_\RQ \frac{\lambda^4}{(1+\lambda|x-\xi_1|)^{10-\mu-\eta}}\\
        &\lesssim \frac{\|\phi_2\|_*}{\lambda^{1+\tau}}
    \end{aligned}
\end{equation*}
with $0<\delta<\min\{2-\eta, 8-\mu\}=2-\eta.$\\
We can observe that
\begin{equation*}
    \begin{aligned}
        |x|^{-\mu}*(\sum_{j=1}^k\Uj)^{\TM}\phi_2&\lesssim\|\phi_2\|_*\int_\RQ\frac{1}{|x-y|^\mu}\frac{\lambda^{\QM}}{(1+\lambda|y-\xi_j|)^{6-\mu+1+\tau}}\, dy\\
        &\lesssim\|\phi_2\|_*\frac{\lambda^{\frac{\mu}{2}}}{(1+\lambda|x-\xi_j|)^{3+\tau-\eta}}.
    \end{aligned}
\end{equation*}
Consequently, 
\begin{equation*}
    \begin{aligned}
        (M)&\lesssim \|\phi_2\|_*\int_\RQ\sum_{j=1}^k\frac{\lambda^{\frac{\mu}{2}}}{(1+\lambda|x-\xi_j|)^{3+\tau-\eta}}\frac{\lambda^{\DM}}{(1+\lambda|x-\xi_1|)^{4-\mu}}\sum_{i=2}^k\frac{\lambda}{(1+\lambda|x-\xi_i|)^{2}}\frac{\lambda}{(1+\lambda|x-\xi_1|)^2}\\
        &+\|\phi_2\|_*\int_\RQ\sum_{j=1}^k\frac{\lambda^{\frac{\mu}{2}}}{(1+\lambda|x-\xi_j|)^{3+\tau-\eta}}\sum_{i=2}^k\frac{\lambda^{\TM}}{(1+\lambda|x-\xi_i|)^{6-\mu}}\frac{\lambda}{(1+\lambda|x-\xi_1|)^2}\\
        &\lesssim \frac{\|\phi_2\|_*}{\lambda^2}\int_\RQ\frac{\lambda^4}{(1+\lambda|x-\xi_1|)^{9+\tau-\mu-\eta}}\, dx +\frac{\|\phi_2\|_*}{\lambda^{2+\delta_1}}\int_\RQ\frac{\lambda^4}{(1+\lambda|x-\xi_1|)^{9+\tau-\mu-\eta-\delta_1}}\, dx \\
        &+\frac{\|\phi_2\|_*}{\lambda^{\delta_1}}\int_\RQ\frac{\lambda^4}{(1+\lambda|x-\xi_1|)^{11+\tau-\mu-\eta-\delta_1}}\, dx+\frac{\|\phi_2\|_*}{\lambda^{2+\delta_1}}\int_\RQ\frac{\lambda^4}{(1+\lambda|x-\xi_1|)^{9+\tau-\mu-\eta-\delta_1}}\, dx\\
        &\lesssim \frac{\|\phi_2\|_*}{\lambda^2}
    \end{aligned}
\end{equation*}
with $0<\delta_1<\min\{6-\mu, 3+\tau-\eta\}$ small enough.\\
To estimate $(N)$ and $(O)$, we observe that 
$$|(\sum_{j=1}^k\Uj)^{\QM}-\U^{\QM}|\lesssim\U^{\TM}\sum_{j=2}^k\Uj+\sum_{j=2}^k\Uj^{\QM}.$$ For this reason, it results that
$$\U^{\TM}\Uj\lesssim \frac{\lambda^{\TM}}{(1+\lambda|x-\xi_1|)^{6-\mu}}\frac{\lambda}{(1+\lambda|x-\xi_j|)^2}\lesssim\frac{1}{\lambda^2|\xi_1-\xi_j|^2}\frac{\lambda^{\QM}}{(1+\lambda|x-\xi_1|)^{6-\mu}}$$  and so $$\U^{\TM}\sum_{j=2}^k\Uj\lesssim\frac{1}{\lambda^2}\frac{\lambda^{\QM}}{(1+\lambda|x-\xi_1|)^{6-\mu}}.$$
Therefore, 
\begin{equation*}
    \begin{aligned}
        |x|^{-\mu}*(\U^{\TM}\sum_{j=2}^k\Uj)&\lesssim \frac{1}{\lambda^2}\int_\RQ\frac{1}{|x-y|^\mu}\frac{\lambda^{\QM}}{(1+\lambda|y-\xi_1|)^{6-\mu}}\, dy\lesssim \frac{1}{\lambda^2}\frac{\lambda^{\frac{\mu}{2}}}{(1+\lambda|x-\xi_1|)^{2-\eta}}
    \end{aligned}
\end{equation*}
and 
\begin{equation*}
    |x|^{-\mu}*\sum_{j=2}^k\Uj^{\QM}\lesssim\sum_{j=2}^k\frac{\lambda^{\frac{\mu}{2}}}{(1+\lambda|x-\xi_j|)^{\mu}}.
\end{equation*}
Finally, we have that
\begin{equation*}
    \begin{aligned}
        (N)&\lesssim \frac{1}{\lambda^2}\int_\RQ\frac{\lambda^{\frac{\mu}{2}}}{(1+\lambda|x-\xi_1|)^{2-\eta}}\frac{\lambda^{\DM}}{(1+\lambda|x-\xi_1|)^{4-\mu}}\frac{\lambda}{(1+\lambda|x-\xi_1|)^{2}}\phi_2\\
        &+\int_\RQ\sum_{j=2}^k\frac{\lambda^{\frac{\mu}{2}}}{(1+\lambda|x-\xi_j|)^{\mu}}\frac{\lambda^{\DM}}{(1+\lambda|x-\xi_1|)^{4-\mu}}\frac{\lambda}{(1+\lambda|x-\xi_1|)^{2}}\phi_2\\
        &\lesssim\frac{\|\phi_2\|_*}{\lambda^2}\int_\RQ\frac{\lambda^3}{(1+\lambda|x-\xi_1|)^{8-\mu-\eta}}\sum_{j=1}^k\frac{\lambda}{(1+\lambda|x-\xi_j|)^{1+\tau}}\, dx \\
        &+\|\phi_2\|_*\int_\RQ\sum_{j=2}^k\frac{\lambda^{3}}{(1+\lambda|x-\xi_j|)^{\mu}}\frac{1}{(1+\lambda|x-\xi_1|)^{6-\mu}}\sum_{i=1}^k\frac{\lambda}{(1+\lambda|x-\xi_i|)^{1+\tau}}\, dx \\
        &\lesssim \frac{\|\phi_2\|_*}{\lambda^2}\int_\RQ\frac{\lambda^4}{(1+\lambda|x-\xi_1|)^{9+\tau-\mu-\eta}}\, dx +\frac{\|\phi_2\|_*}{\lambda^{2+\delta_2}}\int_\RQ\frac{\lambda^4}{(1+\lambda|x-\xi_1|)^{9+\tau-\mu-\eta-\delta_2}}\, dx \\
        &+\frac{\|\phi_2\|_*}{\lambda^{\delta_3}}\int_\RQ\frac{\lambda^4}{(1+\lambda|x-\xi_1|)^{7+\tau-\delta_3}}+\frac{\|\phi_2\|_*}{\lambda^{\delta_4}}\int_\RQ\frac{\lambda^4}{(1+\lambda|x-\xi_1|)^{7+\tau-\delta_4}}\\
        &\lesssim \frac{\|\phi_2\|_*}{\lambda^2}
    \end{aligned}
\end{equation*}
with $0<\delta_2\leq 1+\tau,\, 0<\delta_3<\min\{\mu, 7-\mu+\tau\}$ and $0<\delta_4\leq \min\{1+\tau+\mu, 6-\mu\}.$ We could choose $\delta_3=\delta_4=2.$
\begin{equation*}
    \begin{aligned}
        (O)&\lesssim\int_\RQ\sum_{j=1}^k\frac{\lambda^{\frac{\mu}{2}}}{(1+\lambda|x-\xi_j|)^\mu}\bigg[\U^{\DM}-\bigg(\U^{\DM}+O(\sum_{i=2}^k\Ui^{\DM})\bigg)\bigg]\phi_2\Zo\\
        &\lesssim \|\phi_2\|_*\int_\RQ\sum_{j=1}^k\frac{\lambda^{\frac{\mu}{2}+1}}{(1+\lambda|x-\xi_j|)^{\mu+1+\tau}}\sum_{i=2}^k\frac{\lambda^{\DM}}{(1+\lambda|x-\xi_i|)^{4-\mu}}\frac{\lambda}{(1+\lambda|x-\xi_1|)^2}\, dx \\
        &\lesssim \frac{\|\phi_2\|_*}{\lambda^{\gamma_1}}\int_\RQ\frac{\lambda^4}{(1+\lambda|x-\xi_1|)^{7+\tau-\gamma_1}}+ \frac{\|\phi_2\|_*}{\lambda^{2}}\int_\RQ\frac{\lambda^4}{(1+\lambda|x-\xi_1|)^{5+\tau}}\\
        &\lesssim \frac{\|\phi_2\|_*}{\lambda^{2}}
    \end{aligned}
\end{equation*}
with $0<\gamma<\min\{4-\mu, 3+\mu+\tau\}.$\\
Now, by Proposition \ref{unicità}, we have that
$$\bigg|\int_\RQ\mathcal{L}_2(\phi_1,\phi_2)Z\bigg|=O\bigg(\frac{|\beta|}{\lambda^2}\bigg).$$
\textbf{STEP 3:}\\
Finally, we want to estimate 
$$\bigg|\int_\RQ\mathcal{N}_2(\phi_1,\phi_2)Z(x)\, dx\bigg|.$$
By definition \eqref{N}, it results that
\begin{equation*}
    \begin{aligned}
        &\bigg|\int_\RQ\mathcal{N}_2(\phi_1,\phi_2)Z(x)\, dx\bigg|\lesssim\underbrace{\int_\RQ|g(V+\phi_2)-g(V)-g'(V)\phi_2|Z\, dx}_{(P)}+\\
        &+\underbrace{|\beta|\int_\RQ\mathcal{U}^2|\phi_2|Z\, dx}_{(Q)} +\underbrace{|\beta|\int_\RQ\mathcal{U}V|\phi_1|Z\, dx }_{(R)}+\underbrace{|\beta|\int_\RQ\mathcal{U}|\phi_1||\phi_2|Z\, dx }_{(S)}+\\
        &+\underbrace{|\beta|\int_\RQ V|\phi_1|^2Z\, dx}_{(T)}+\underbrace{|\beta|\int_\RQ|\phi_1|^2|\phi_2|Z\, dx }_{(V)}.
    \end{aligned}
\end{equation*}
Reasoning as in Lemma 2.3 of \cite{CV}, we can deduce that $$(P)=O\bigg(\frac{|\beta|}{\lambda^2}\bigg).$$
Now, 
\begin{equation*}
    \begin{aligned}
        (Q)&\lesssim|\beta| \|\phi_2\|_*\int_\RQ\frac{1}{(1+|x|^2)^2}\sum_{j=1}^k\frac{\lambda}{(1+\lambda|x-\xi_j|)^{1+\tau}}\frac{\lambda}{(1+\lambda|x-\xi_j|)^{2}}\, dx\\
        &\lesssim |\beta| \|\phi_2\|_*\frac{1}{\lambda^{1+\tau}}\int_\RQ\frac{1}{(1+|x|^2)^2}\sum_{j=1}^k\frac{1}{|x-\xi_j|^{3+\tau}}\lesssim\frac{|\beta|}{\lambda^{1+\tau}}\|\phi_2\|_*
    \end{aligned}
\end{equation*}
\begin{equation*}
    \begin{aligned}
        (R)&\lesssim |\beta|\|\phi_1\|_*\int_\RQ\frac{1}{(1+|x|^2)}\sum_{j=1}^k\frac{\lambda}{(1+\lambda|x-\xi_j|)^2}\frac{\lambda}{(1+\lambda|x-\xi_j|)^{1+\tau}}\frac{\lambda}{(1+\lambda|x-\xi_j|)^2}\, dx\\
        &\lesssim|\beta|\|\phi_1\|_*\int_\RQ \frac{1}{(1+|x|^2)}\sum_{j=1}^k\frac{\lambda^3}{(1+\lambda|x-\xi_j|)^{5+\tau}}\, dx \\
        &\lesssim|\beta|\|\phi_1\|_*\frac{1}{\lambda}\int_\RQ\frac{1}{(1+|\frac{x}{\lambda}+\xi_j|^2)}\frac{1}{(1+|x|)^{5+\tau}}\, dx
        \lesssim \frac{|\beta|}{\lambda}\|\phi_1\|_*
    \end{aligned}
\end{equation*}
\begin{equation*}
    \begin{aligned}
        (S)&\lesssim|\beta|\|\phi_1\|_*\|\phi_2\|_*\int_\RQ\frac{1}{(1+|x|^2)}\sum_{j=1}^k\frac{\lambda^2}{(1+\lambda|x-\xi_j|)^{2+2\tau}}\frac{\lambda}{(1+\lambda|x-\xi_j|)^{2}}\, dx \\
        &\lesssim|\beta|\|\phi_1\|_*\|\phi_2\|_*\int_\RQ\frac{1}{(1+|x|^2)}\sum_{j=1}^k\frac{\lambda^3}{(1+\lambda|x-\xi_j|)^{4+2\tau}}\, dx \\
        &\lesssim |\beta|\|\phi_1\|_*\|\phi_2\|_*\frac{1}{\lambda}\int_\RQ\frac{1}{(1+|\frac{x}{\lambda}+\xi_j|^2)}\frac{1}{(1+|x|)^{4+2\tau}}\, dx \lesssim\frac{|\beta|}{\lambda}\|\phi_1\|_*\|\phi_2\|_*
    \end{aligned}
\end{equation*}
\begin{equation*}
    \begin{aligned}
        (T)&\lesssim|\beta|\|\phi_1\|_*^2\sum_{j=1}^k\int_\RQ\frac{\lambda^{4}}{(1+\lambda|x-\xi_j|)^{6+2\tau}}\, dx \lesssim|\beta|\|\phi_1\|_*^2
    \end{aligned}
\end{equation*}
and 
\begin{equation*}
    \begin{aligned}
        (V)&\lesssim|\beta|\|\phi_1\|_*^2\|\phi_2\|_*\sum_{j=1}^4\int_\RQ\frac{\lambda^{4}}{(1+\lambda|x-\xi_j|)^{5+3\tau}}\, dx \lesssim |\beta|\|\phi_1\|_*^2\|\phi_2\|_*
    \end{aligned}
\end{equation*}
Therefore, by Proposition \ref{unicità}, we have that
$$\bigg|\int_\RQ\mathcal{N}_2(\phi_1,\phi_2)Z(x)\, dx\bigg|=O\bigg(\frac{|\beta|}{\lambda^2}\bigg).$$
\end{proof}
\begin{proof}[Proof of Theorem \ref{mainT}]
By Proposition \ref{unicità} and Proposition \ref{ridotto}, it is enough to fix the parameter $\lambda$ such that $\frac{1}{\lambda}\in \bigg(0, e^{-\frac{1}{\sqrt{|\beta|}}}\bigg)$ and so $\mathfrak{c}(\lambda)=0$. In fact, we can choose 
$$\frac{1}{\lambda}=e^{-d_\beta}\quad\hbox{with}\quad d_\beta:=\frac{1}{|\beta|}\frac{\mathfrak{a}}{\mathfrak{b}}+o(1)>0,$$ so that
$$-\mathfrak{a}-\mathfrak{b}\beta d_\beta+o(1)=\underbrace{-\mathfrak{a}+\mathfrak{b}\beta\ln{\bigg(\frac{1}{\lambda}\bigg)}+o(1)}_{\mathfrak{c(\lambda)}}=0.$$
    
\end{proof}
\section{Useful estimates}
In \cite{DAI} it was shown that (see formula (37))
\begin{equation}\label{R}
    \int_\RN\frac{1}{|x-y|^{2s}}\bigg(\frac{1}{1+|y|^2}\bigg)^{N-s}\, dy =\mathcal{I}(s)\bigg(\frac{1}{1+|x|^2}\bigg)^s,\, 0<s<\frac{N}{2}
\end{equation}
where 
$$\mathcal{I}(s):=\frac{\pi^{\frac{N}{2}}\Gamma(\frac{N-2s}{2})}{\Gamma(N-s)},\quad \Gamma(s)=\int_0^{+\infty}x^{s-1}e^{-x}\, dx, \quad s>0.$$
Hence, as in \cite{CV}, it follows that
\begin{remark}\label{REM}
    $$|x|^{-\mu}*\Uj^{\QM}=\mathfrak{\sigma}_{4,\mu}\bigg(\frac{\lambda}{(1+\lambda^2|x-\xi_j|^2}\bigg)^{\frac{\mu}{2}}$$ where $\mathfrak{\sigma}_{4,\mu}:=\alpha_{4,\mu}\mathcal{I}\big(\frac{\mu}{2}\big).$
\end{remark}
Now, we report some useful estimates (see Lemma 5.1 and Lemma 5.2 of \cite{CV}).
\begin{lemma}\label{LEM1}
    For each fixed $i$ and $j$ with $i\not=j$ we let 
    $$f_{i,j}(x)=\frac{1}{(1+|x-\xi_i|)^\alpha}\frac{1}{(1+|x-\xi_j|)^\beta}$$ where $\alpha,\beta \geq 1$ are two constants. Then, for any $0<\sigma\leq\min\{\alpha,\beta\},$ there is a constant $C>0$ such that
    $$f_{i,j}(x)\leq\frac{C}{|\xi_i-\xi_j|^\sigma}\bigg(\frac{1}{(1+|x-\xi_i|)^{\alpha+\beta-\sigma}}+\frac{1}{(1+|x-\xi_j|)^{\alpha+\beta-\sigma}}\bigg).$$
\end{lemma}
\begin{lemma}\label{LEM2}
    There is a constant $C>0$ such that
    $$\int_\RQ\frac{1}{|y|^\mu}\frac{1}{(1+|\lambda(x-\xi)-y|)^{\alpha+\eta}}\, dy\leq C\frac{1}{(1+\lambda|x-\xi|)^{\alpha-4+\mu}}$$ for $\alpha\geq4-\mu$ and $\eta>0.$
\end{lemma}
\section*{Acknowledgments}
The author would like to thank her supervisor, Professor Giusi Vaira, for her invaluable guidance and support throughout this research. \\
Additionally, this work has been supported by the GNAMPA project "Singolarità, nonlinearità e interazioni non locali: nuovi approcci analitici per equazioni differenziali e sistemi", CUP E53C25002010001.


\begin{thebibliography}{99}
\bibitem{CV} 
S. Caputo, G. Vaira, \textit{Infinitely many non-radial solutions to a critical Choquard equation}, https://arxiv.org/abs/2507.15747;
\bibitem{CL}
W. Chen, C. Li, B. Ou, \textit{Classification of solutions for an integral equation}, Comm. Pure Appl. Math. 59 (2006), 330–343;
\bibitem{CLO}
W. Chen, C. Li, B. Ou, \textit{Classification of solutions for a system of integral equations}, Comm. PDEs 30 (2005), 59–65;
\bibitem{CMP}
H. Chen, M. Medina, A. Pistoia, \textit{Segregated solutions for a critical elliptic system with a small interspecies repulsive force}, https://arxiv.org/abs/2203.10990;
\bibitem{DAI}
W. Dai, J. Huang, Y. Qin, B. Wang, Y. Fang, \textit{Regularity and classification of solutions to static Hartree equations involving fractional Laplacians}, Discrete Contin. Dyn. Syst., 39 (2019), 1389-1403;
\bibitem{DMPP}
M. del Pino, M. Musso, F. Pacard, A. Pistoia, \textit{Large energy entire solutions for the Yamabe equation}, J. Differential Equation 251 (2011), 2568-2597;
\bibitem{DMPPP}
M. del Pino, M. Musso, F. Pacard, A. Pistoia, \textit{Torus action on $S^N$ and sign changing solutions for conformally invariant equations}, Ann. Sc. Norm. Super. Pisa CI. Sci. 12 (2013), 209-237;
\bibitem{DY}
L. Du, M. Yang, \textit{Uniqueness and nondegeneracy of solutions for a critical nonlocal equation}, Discrete Contin. Dyn. Syst., 39 (2019), 5847–5866;
\bibitem{F}
H. Fr\"ohlich, \textit{Theory of electrical breakdown in ionic crystal}, Proc. Roy. Soc. Ser. A 160 (1937), no. 901, 230-241;
\bibitem{FF}
H. Fr\"ohlich, \textit{Electrons in lattice fields}, Adv. in Phys. 3 (1954), no. 11;
\bibitem{GM}
F. Gao, V. Moroz, M. Yang, S. Zhao, \textit{Construction of infinitely many solutions for a critical Choquard equation via local Pohozaev identities}, https://arxiv.org/abs/2206.14958;
\bibitem{GH}
Y. Guo, T. Hu, S. Peng, W. Shuai, \textit{Existence and uniqueness of solutions for Choquard equation involving Hardy-Littlewood-sobolev critical exponent}, Calc. Var. PDEs 58 (2019) 128, 34 pp.2;
\bibitem{LEI}
Y. Lei, \textit{Liouville theorems and classification results for a nonlocal Schrödinger equation}, Discrete and Cont. Dyn. Syst. 38 (2018), 5351–5377;
\bibitem{LI}
X. Li, C. Liu, X. Tang, G. Xu, \textit{Nondegeneracy of positive bubble solutions for generalized energy- critical Hartree equations}, to appear on IMRN;
\bibitem{LL}
E. Lieb, M. Loss, \textit{Analysis}, Graduate Studies in Mathematics. Providence: Amer. Math. Soc., (2001);
\bibitem{MV}
V. Moroz, J. Van Schaftingen, \textit{A guide to the Choquard equation}, Journal of Fixed Point Theory and Applications, 19(1), 773-813, (2017);
\bibitem{P}
S. Pekar, \textit{Untersuchung \"uber die Elektronentheorie der Kristalle}, Akademie Verlag, Berlin, (1954);
\bibitem{RB}
R. Ruffini, S. Bonazzola,\textit{Systems of self-gravitating particles in general relativity and the concept of an equation of state}, Physical Review, 187(5), (1969); 
\end{thebibliography}
\end{document}